\documentclass[final,3p,times]{elsarticle}

\usepackage{amssymb}
\usepackage{amsmath}
\usepackage{amsthm}
\usepackage{mathtools}
\usepackage{bm}
\usepackage{subcaption}
\usepackage{xcolor}
\usepackage{graphicx}

\newtheorem{theorem}{Theorem}[section]
\newtheorem{lemma}[theorem]{Lemma}
\newtheorem{corollary}[theorem]{Corollary}
\newtheorem{definition}[theorem]{Definition}
\newtheorem{example}[theorem]{Example}
\newtheorem{remark}[theorem]{Remark}
\newtheorem{proposition}[theorem]{Proposition}

\newcommand{\R}{\mathbb{R}}
\newcommand{\diag}{\mathrm{diag}}
\newcommand{\rd}{\mathrm{d}}
\newcommand{\tr}{\mathrm{tr}}

\graphicspath{{figs/}}

\journal{...}

\begin{document}

\begin{frontmatter}

\title{Explicit invariant-preserving integration of differential equations using homogeneous projection}

\author[sintef]{Benjamin Kwanen Tapley}
\ead{ben.tapley@sintef.no}

\affiliation[sintef]{organization={Department of Mathematics and Cybernetics, SINTEF Digital},
            city={Oslo},
            country={Norway}}

\begin{abstract}
We develop a general framework for numerically solving differential equations while preserving invariants. As in standard projection methods, we project an arbitrary base integrator onto an invariant-preserving manifold, however, our method exploits homogeneous symmetries to evaluate the projection exactly and in closed form. This yields explicit invariant-preserving integrators for a broad class of nonlinear systems, as well as pseudo-invariant-preserving schemes capable of preserving multiple invariants to arbitrarily high precision. The resulting methods are high-order and introduce negligible computational overhead relative to the base solver. When incorporated into adaptive solvers such as Dormand–Prince 8(5,3), they provide error-controlled, invariant-preserving, high-order time-stepping schemes. Numerical experiments on double-pendulum and Kepler ODEs as well as semidiscretised KdV and Camassa-Holm PDEs demonstrate substantial improvements in both accuracy and efficiency over standard approaches.
\end{abstract}


\end{frontmatter}

\section{Introduction}
Let $x(t)\in\mathbb{R}^n$ denote the solution to an autonomous ordinary differential equation (ODE) $\dot{x} = f(x),$
where $f(x)\in C^{p+1}(U,\R^{n})$ is a smooth vector field on compact $U\subset\R^{n}$ with exact flow\footnote{That is, $\frac{\rd}{\rd t}\varphi_t(x)=f(\varphi_t(x))$, $\varphi_s\circ\varphi_t(x)=\varphi_{s+t}(x)$ and $\varphi_0(x)=x$} $\varphi_t$. Then a function $H\in C^{1}(U,\R)$ is called an \emph{invariant} of $f(x)$ if $\dot{H}(x)=\nabla H(x)^\top f(x)=0$. That is, $H(\varphi_t(x))=H(x)$ is constant. The preservation of the invariants along solutions is usually an expression of deeper physical structure underlying the system, like a manifestation of symmetry through Noether's theorem, a conservation law or a physical constraint \cite{hairer2006geometric}. Respecting these features in simulations is therefore important for maintaining physical fidelity, long-term accuracy and stability of numerical solutions. However, achieving this in practice is not always straight forward. 

Standard numerical methods, such as Runge-Kutta methods are only known to preserve invariants for a very limited set of functions. For example, they always preserve rational affine invariants of the form $H(x)=\ell_1(x)/\ell_2(x)$, for affine forms $\ell_i(x)=a_i^Tx+\alpha_i$, where $a_i\in\R^n$, $\alpha_i\in\R$ and $\ell_2(x)\ne 0$. This is due to their recently discovered ability to preserve affine weak invariants \cite{tapley2023preservation}, although the affine case $\ell_2(x)=1$ has been known for a long time \cite{hairer1987solving}. For quadratic invariants, only symplectic Runge-Kutta methods are able to preserve these in general \cite{cooper1987stability, sanz2018numerical}, which are necessarily implicit and so require solving a system of nonlinear equations at each time step. For any invariant more complicated than these two cases, such as polynomials and general functions in $C^1(\R^n)$, one must resort to specialized invariant-preserving methods, of which there are several. These include discrete gradient methods \cite{mclachlan1999geometric}, projection methods \cite{norton2015discrete}, Hamiltonian boundary value methods \cite{brugnano2010hamiltonian}, auxiliary variable methods \cite{tapley2022geometric}, relaxation methods \cite{ketcheson2019relaxation, biswas2023multiple} and so forth. While most invariant-preserving methods are implicit, several linearly implicit methods are known to preserve invariants and only require solving a linear system at each time step. Examples include Kahan's method for certain cubic invariants \cite{celledoni2012geometric}, multi-step polarised methods for polynomial invariants \cite{celledoni2015discretization}. 

Although many invariant-preserving methods are often formulated for solving ODEs, they can be easily extended to semi-discretised partial differential equations (PDEs) and gradient flow problems. For example auxiliary variable methods \cite{shen2018scalar,cai2025explicit, zhang2020novel, jiang2022improving, obbadi2025stable, hong2025auxiliary}, Lagrange multiplier approaches \cite{cheng2025unique}, discrete gradients \cite{dahlby2011general}, Kahan methods \cite{eidnes2021linearly, uzunca2023linearly} and so forth. Many of these methods have the benefit of being linearly implicit, such as auxiliary variable methods, Kahan methods and polarisation methods. While, solving a linear system at each time step is often an acceptable price to pay for the benefit invariant preservation, it can still be expensive for large systems and are not usually applicable when needing to preserve a general invariant in the original phase space. Explicit methods, when all else equal, are preferred for their speed and simplicity.

While many of the aforementioned frameworks do an excellent job in preserving an invariant exactly, many PDEs of interest come with several conservation laws like conservation of mass, momentum, energy and so forth. A major challenge when it comes to numerically solving such PDEs is how to represent them in finite dimensions, a process known as semidiscretisation, such that the discretised system possesses similar conservation laws. It is not too difficult to find energy-preserving semidiscretisations that yield a system of ODEs that possess an exact conservation law, see for example \cite{celledoni2012preserving}. However, preserving one discrete conservation law under semidiscretisation often results in other conservation laws being lost in the process. This can result in instabilities due to blow up of high frequency modes that comprimise the simulation \cite{offen2019symplectic_clebsch, coppola2019discrete}. 

\subsection{Key contributions.}
In this paper, we develop a projection framework for deriving explicit invariant-preserving methods for a large class of invariant functions, as well as pseudo-invariant-preserving methods for a wider class of functions. The main contributions are summarised as follows:
\begin{enumerate}
\item \textit{Linear homogeneous projection methods.} We develop a method that exploits homogeneous symmetries of $H(x)$ to define a group action $\psi_t(x)$ that pushes forward the base step $\Phi_h(x)$ of order-$p$ to the invariant-preserving manifold, resulting in an order-$p$ composition method $\widehat{\Phi}_h(x) = \psi_h\circ\Phi_h(x)$. This has the advantage that we can achieve exact invariant-preservation without solving any nonlinear equations, provided $\psi_t(x)$ is known in closed form. We show that such a $\psi_t(x)$ can be evaluated as a cheap linear scaling of the state $x$ for a large class of invariants.
\item \textit{Pseudo-nonlinear homogeneous projection methods.} If $\psi_t(x)$ is not attainable in closed, we can instead approximate it with $r$ steps of another one-step integrator $\Psi_h(x)$ of order-$q$. This results in an order-$p$ method $\widehat{\Phi}_h(x) = \left(\Psi_h\right)^{r}\circ\Phi_h(x)$ that preserves invariants to an accuracy of $O(h^{(p+1)(q+1)^r})$ when a base method of accuracy $O(h^{p+1})$ is being used.
\item \textit{Error-controlled, open-source implementation.} We integrate our methods in popular adaptive implementations using \texttt{scipy.integrate\_ivp}, such as Dormand-Prince 8(5,3), yielding error-controlled invariant-preserving high-order time-stepping methods. The code is available at the author's github repository. We show that the methods yield significant improvements to accuracy and speed over baseline methods like the standard adaptive methods.
\item \textit{Regularization of semidiscrete PDEs.} We show empirically that even though a semi-discretised PDE may not exactly preserve discrete versions of the continouous invariants, a pseudo-invariant-preserving integration method can still improve the stability, physical fidelity and qualitative behaviour of the solution.
\end{enumerate}
The rest of the paper is structured as follows. In section 2 we review related work on invariant-preserving methods, adaptive geometric integrators and projection methods. In section 3 we present the theory of homogeneous projection methods for invariant-preservation. In sections 4 and 5 we discuss how to obtain homogeneous symmetry projections $\psi_h$. In section 6 we present numerical experiments on ODEs and address integrating semidiscretised PDEs in section 7. 
\section{Related work}
\paragraph{Symplectic vs. energy-preserving methods} Hamiltonian systems possess an invariant function called the Hamiltonian (or energy) as well as an invariant two-form called the symplectic form. Due to a theorem by Marsden and Ge \cite{zhong1988lie}, only time reparameterisations of the exact solution can simultaneously preserve energy and symplectic structure. So when developing a numerical method one must therefore make a choice between one or the other (or neither). It is not obvious in general which structure is more important to preserve in the numerical solution and the answer is probably very problem dependent. There is a case for both. 

Symplectic integrators preserve phase-space volume and possess many properties that make them favourable for long-time simulation \cite{leimkuhler2004simulating}. Symplectic methods also have rich theory and it can be shown that all symplectic methods are the exact flow of a nearby (modified) Hamiltonian system, meaning they preserve an energy. They are also known to have fast implementations available through splitting methods \cite{mclachlan2002splitting} when the Hamiltonian is separable and can be easily extended to systems contain simple (non-symplectic) forcing terms \cite{tapley2022computational,tapley2019novel}. In such situations, symplectic methods tend to be favoured.

However, despite their popularity, a major notable challenge with symplectic methods is their difficulty with their use in adaptive methods. Adaptive step control in symplectic integrators presents a significant obstacle because standard controllers use past-step error estimates, which breaks the symplectic structure, time-reversal symmetry and the backward-error modified-Hamiltonian picture \cite{gladman1991symplectic}. For this reason, to even retain time-reversal symmetry (a weaker property than symplecticity) specialized controllers are required \cite{hairer2005explicit}. This severely limits the use of symplectic methods as adaptivity is often necessary for problems with multiple time scales or close encounters, such as $n$-body problems, chemical reactions, chaotic systems and so forth. 

On the other hand, invariant-preserving methods have no such issues with their use in adaptive methods. This means that invariant-preserving methods can be used with adaptive step control without any special considerations. This is a significant advantage of invariant-preserving methods over symplectic methods that we will explore. Invariants also have physical interpretations in many systems, such as conservation laws, symmetries or constraints. Preserving energy exactly can therefore be important for the physical fidelity of the simulation. 

Furthermore, invariant-preserving methods can be applied to a broader class of systems than symplectic methods, which are mainly only used for Hamiltonian systems or perturbations thereof. Invariant-preserving methods can be applied to any ODE with invariants, such as Poisson systems, constrained mechanics, semi-discretised PDEs and so forth. Furthermore, invariant-preserving methods can be easily modified for systems possessing Lyapunov functions, dissipation, external forces, control terms, gradient flows and so forth. In such situations, one can easily control how the methods dissipate or grow the invariants at the correct rate \cite{mclachlan1999geometric}.
In this case, and especially in gradient flows of PDEs that diminish an energy, invariant-preserving methods are usually easy to adapt to preserve the dissipative structure \cite{fu2024energy}. 

\paragraph{Adaptive step size geometric methods}  
Despite the above difficulties, there do exist good frameworks for reversible step-size control methods. One such framework of Hairer and Söderlind \cite{hairer2005explicit} introduces a time reparametrization $\rd t = \varepsilon \rd\tau / \rho(\tau)$, where $\rho(\tau)$ and $x$ evolve according to an augmented, reversible system
\begin{equation}\label{hairer-soderlind}
x' = \frac{\varepsilon}{\rho(\tau)} f(x), \quad \rho' = G(x),
\end{equation}
where $x' = \frac{\rd x}{\rd \tau}$. Discretizing this extended system with a symmetric splitting method with stepsize $\tau_{n+1}-\tau_{n}=\varepsilon$ (i.e., a fixed step size method in $\tau$) yields the variable step size method in $t$ of step size $t_{n+1}-t_n = h_{1/2}=\varepsilon/\rho(\tau_{n+1/2})$. The integration in physical time $t$ can be done using any symmetric method $\Phi_{h_{1/2}}(x)$, such as a high order splitting scheme. This will be used for comparison purposes in the numerical experiments section. Hence, the auxiliary variable $\rho(\tau)$ acts as a step-density function, concentrating integration points where the state changes rapidly and stretching them where the motion is smooth. This results in an adaptive scheme that remains time-reversible and retains a lot of good qualitative properties of symplectic methods, unlike conventional error-based controllers that lose such properties. Similar adaptive methods are explored using more exotic time transformations, see e.g., \cite{blanes2005adaptive, blanes2012explicit} or variational approaches such as \cite{duruisseaux2021adaptive}.

The main advantages of time-reparameterisations such as the above are reversibility, compatibility with geometric integrators such as splitting methods, and the ability to adapt step sizes in a structure-preserving way, making it particularly effective for systems with varying time scales. Its limitations are that it is problem-dependent requiring a suitable control function $G(y)$, parameters such as $\varepsilon$ must be carefully tuned to avoid instability or excessive step variation. Most importantly, they do not directly control the local error, meaning that true error control is not guaranteed.

\paragraph{Extended phase space and auxiliary variable methods}
Many popular methods augment the phase space with auxiliary variables to reform the system into one that is easier to solve. One such method is the invariant energy quadrisation \cite{zhang2020novel} and scalar auxiliary variable methods \cite{shen2018scalar, bo2022arbitrary}, which reformulates a complicated invariant into a quadratic one on an extended ODE, which can be solved using linearly implicit methods. In \cite{bilbao2023explicit} they show that such methods can be made explicit and can improve stability over a symplectic method, however do not improve accuracy by much. As noted in \cite{zhang2021remark} auxiliary variable and energy quadrisation methods do not actually preserve the true invariant in the original phase space variables, but rather the modified energy on the extended space. This is due to the fact that the constraints introduced by the auxiliary variables are broken. This problem is addressed in \cite{chen2022novel, tapley2022geometric} by also preserving the auxiliary variable constraints exactly, resulting in true invariant-preserving methods. However, these methods are fully implicit and only work for polynomial invariants. We remark also that pseudo-symplectic methods have also been developed using extended phase space methods \cite{tao2016explicit,mclachlan2024runge}.

\paragraph{Projection methods and discrete gradients}
Discrete gradient methods are one of the earliest invariant preserving methods for preserving an arbitrary number of invariants. Despite their popularity, they face a number of limitations, such as being implicit, difficulty achieving high order methods \cite{eidnes2022order}, and difficulty acheiving simultaneous preservation of multiple invariants \cite{mclachlan1999geometric}. Although the last two points are theoretically possible, the resulting methods are often complicated to derive, problem dependent and computationally expensive. 

Projection methods are a useful and powerful tool for constructing numerical methods that preserve invariants. Surprisingly, projection has been shown to be a subset of discrete gradient methods \cite{norton2015discrete}, which together have been used copiously to construct conservative methods for ODEs and PDEs \cite{dahlby2011general, calvo2006preservation, kojima2016invariants, jiang2021explicit} to name a few. Projection methods are constructed by first stepping with a ``base method'' $\Phi_h(x)$, such as any one-step method, then the ``projected method'' is defined as follows. 
\begin{definition}[Projection method]\label{def:projection method}
  Let $\Phi_h(x)$ be a one-step method, then a projection method is given by
\begin{equation}\label{eq:proj method}
  \widehat{\Phi}_h(x) =  \Phi(x) + \lambda v(x),
\end{equation}
where $\lambda\in\R$ is chosen such that 
\begin{equation}\label{eq:proj cond}
  H(\Phi_h(x) + \lambda v(x)) = H(x)
\end{equation}
 and $v(x)\in\R^n$ is some direction vector.
\end{definition}  
The choice of $v(x)$ is somewhat arbitrary, but common choices include $v(x) = \nabla H(x)$, $v(x) = f(x)$ or $v(x) = f(\Phi_h(x))$. The choice of $v(x)$ can affect the stability and accuracy of the method, but in practice, the method is not too sensitive to a sensible choice. A drawback of projection methods is that one must solve the nonlinear equation \eqref{eq:proj cond} for $\lambda$ at each step, which can be expensive. 

If one prefers the speed of explicit methods and is okay with near preservation of invariants, then an alternative approach is to use pseudo-invariant-preserving methods, that preserve invariants to a higher order than the numerical method being used. As has been shown, such methods are equivalent to projection methods by solving for $\lambda$ with a single Newton iteration and are shown to be preserve invariants to an accuracy of $O(h^{2(p+1)})$ when a base method of accuracy $O(h^{p+1})$ is being used \cite{cai2020explicit}. One advantage of the projection and pseudo-projection methods is that they retain the same order of accuracy as the base method being used, meaning that achieving high order accuracy is trivial.


\section{Homogeneous projection} 
\subsection{Motivating example: a planar separable polynomial Hamiltonian}
We begin with a simple example to illustrate the method. Consider a Hamiltonian system with energy $H(q,p) = \frac{1}{\alpha}p^\alpha + \frac{a}{\beta}q^\beta$ for integers $\alpha, \beta > 0$. The equations of motion are given by
\begin{equation}\label{hamiltonian}
    \dot{q} = p^{\alpha-1}, \quad \dot{p} = -aq^{\beta-1}.
\end{equation}
Consider the forward Euler method applied to \eqref{hamiltonian} with step size $h>0$
\begin{equation}
    q' = q + h p^{\alpha-1}, \quad p' = p - h a q^{\beta - 1}.
\end{equation}
We now apply a map $\psi$ called ``homogeneous projection'' (defined later), which is given by
\begin{equation}\label{proj fe}
        \psi(q',p')=(\lambda^{\frac{1}{\beta}}q',\lambda^{\frac{1}{\alpha}}p'),
\end{equation}
for some free parameter $\lambda$. Calculating the energy gives
\begin{equation}
    H(\psi(q',p')) = H(\lambda^{\frac{1}{\beta}} q',\lambda^{\frac{1}{\alpha}}p') = \frac{\lambda}{\alpha} p'^\alpha + \frac{a\lambda}{\beta} q'^\beta = \lambda H(q',p').
\end{equation} 
The important point here is that the energy scales by $\lambda$ under the pullback of $\psi$, which is a due to $H$ being a homogeneous function. We can therefore leverage this property to select $\lambda$ such that this exactly matches the desired energy value. Defining $\lambda := \frac{H(q,p)}{H(q',p')}$ gives 
\begin{equation}
    H(\psi(q',p')) = H(q,p),
\end{equation}
hence the method $\psi(q',p')$ preserves the energy exactly and we have used no nonlinear solves. Note that $ \lambda = 1 + O(h^2)$ hence the method also retains its first order accuracy. Homogeneous functions and related concepts have played a key role in the mathematics and physics dating back to Euler. For example in control theory \cite{polyakov2020generalized,rosier1992homogeneous} where homogeneous Lyapunov functions are used to show stability of ODEs or in discrete integrable systems \cite{celledoni2019using, celledoni2019detecting} where they sometimes go by Darboux polynomials or weak invariants. However, their use in numerical methods have so far been limited. 

The rest of this section will formalise the above idea into a general numerical method. In section \ref{sec:linear homogeneous symmetries} we which functions $H(x)$ possess a \textit{linear} projection $\psi$ that can be evaluated cheaply and in closed form. In section \ref{sec:nonlinear homogeneous symmetries} we extend the method a broader class of functions where $\psi$ is a non-linear action. 
\subsection{General framework}
We now define a general notion of homogeneity, with respect to a flow, and discuss more familiar specific examples of homogeneity later.
\begin{definition}[Homogeneous function]
  Let $\psi_t$ be the flow of $g(x)\in C^{q+1}(U)$. Then a function $H \in C^1(U)$ is said to be homogeneous of degree $k$ with respect to $\psi_t$ if
\begin{equation}\label{hom cond}
    H(\psi_t(x))=e^{tk}H(x).
\end{equation}
Its infinitesimal condition known as Euler's equation for homogeneous functions is
\begin{equation}\label{euler eq}
    g(x)^T\nabla H(x)=k H(x)
\end{equation}
where $g(x)$ is called the infinitesimal generator of $\psi_t$.
\end{definition}
That is, the flow $\psi_t(x)$ modifies the invariant $H(x)$ by a multiplicative factor $e^{tk}$. An important point here is that $\psi_t(x)$ depends on $H(x)$ only, and is not necessarily related to the underlying ODE $\dot{x}=f(x)$.

Letting $\psi_t(x) = e^{t} x$ and $\lambda=e^{t}$ (otherwise known as isotropic scaling) recovers the most basic notion of a homogeneous function studied by Euler: functions satisfying $H(\lambda x) = \lambda^k H(x)$, with the infinitesimal condition $x^T\nabla H(x)=k H(x)$.

Let $\Phi_h(x)$ be an approximation to $\varphi_h$. We say that it is order-$p$ if for some constant $C$ we have
$$
\sup_{x\in U}\|\Phi_h(x)-\varphi_h(x)\|\le C\,h^{p+1},
$$
as $h\to 0$ for compact $U\subset\R^n$. As $\Phi_h(x)$ is not necessarily invariant-preserving, we have the energy-drift $H(\Phi_h(x)) = H(x) + \mathcal{O}(h^{p+1}).$ The idea is that if $H$ is homogeneous with respect to some $\psi_t$, then we can construct explicit post-processing corrections that exactly restore desired target values via the following method.
\begin{definition}[Homogeneous projection method]\label{def:hom proj method}
    Let $H$ be homogeneous with respect to $\psi_t$ according to \ref{def:hom proj method} and assume $H(x)H(\Phi_h(x))>0$, that is, $H(x)$ and $H(\Phi_h(x))$ are both non-zero and of the same sign. Then the method
    \begin{equation}
      \widehat\Phi_{h}:=\psi_{s}\circ\Phi_h(x),\quad\text{where} \quad s:= \frac{1}{k}\log\left(\frac{H(x)}{H(\Phi_h(x))}\right)
    \end{equation}
    is called the homogeneous projection of $\Phi_h(x)$.
\end{definition}
We note that $s$ can always be calculated purely in terms of $H(x)$ and $H(\Phi_h(x))$, hence requires no nonlinear solves. Furthermore, if $\Phi_h(x)$ is an order-$p$ method, we can show that the correction $\psi_t$ is small enough to not affect the order of the method meaning $\widehat\Phi_{t}(x)$ maintains the order of $\Phi_h(x)$. This is summarised by the following, which resembles many similar results on order retention for projection methods \cite{norton2015discrete, calvo2006preservation}.  
\begin{theorem}[Invariant-preservation and order-retention]\label{thm:hom proj method}
    The homogeneous projection method $\widehat{\Phi}_h$ satisfies the following:
    \begin{enumerate}
    \item $H(\widehat\Phi_{h}(x)) = H(x)$ (invariant-preserving), and 
    \item $\|\widehat\Phi_{h}(x)-\varphi_h(x)\| = \mathcal{O}(h^{p+1})$ (same order as $\Phi_h$).
    \end{enumerate}
\end{theorem}
\begin{proof}
    We have
    \begin{align*}
        H(\widehat\Phi_{h}(x)) &= H(\psi_{s}\circ\Phi_h(x)) = e^{ks}H(\Phi_h(x)).
    \end{align*}
    Choosing $s = \frac{1}{k}\log\left(\frac{H(x)}{H(\Phi_h(x))}\right)$ solves $e^{ks}H(\Phi_h(x)) = H(x)$, which yields (1). To show (2), denote the energy error as $\delta_h :=  H(\Phi_h (x))-H(\varphi_h(x))$. Since $H$ is $C^1$ and $\Phi_h$ has local order $p$, there exists some $L_H>0$ on a compact neighborhood such that
    $$
      |\delta_h| \le L_H\,\|\Phi_h (x)-\varphi_h(x)\|\;=\;\mathcal{O}(h^{p+1}).
    $$
    Insert the defintion for $\delta_h $ into $s$ to obtain
    $$
      s = -\,\frac{1}{k}\log\!\left(1+\frac{\delta_h}{H(x)}\right),
    $$
    to which we apply the Taylor truncation  $\log(1+u)=u+\mathcal{O}(u^2)$ to get
    $$
      |s|\;\le\; C_s\,|\delta_h|\;=\;\mathcal{O}(h^{p+1})
    $$
    for some constant $C_s$ independent of $h$. Next, assuming on compact neighbourhoods we have $\|g(\xi)\|\le C_g$ meaning $\|\psi_s(z)-z\|\le C_g |s|$. Applying the triangle inequality yields
    $$
      \|\psi_s\circ\Phi_{h}(x)-\varphi_h(x)\|
      \;\le\; \underbrace{\|\Phi_h(x)-\varphi_h(x)\|}_{=\mathcal{O}(h^{p+1})}
      \;+\; \underbrace{\|\psi_{s}(\Phi_h(x))-\Phi_h(x)\|}_{\le C_g |s|=\mathcal{O}(h^{p+1})}
      \;=\;\mathcal{O}(h^{p+1}),
    $$
    which proves (2).
\end{proof}

Homogeneous projection relies on the existence of some flow $\psi_t$ that satisfies \eqref{hom cond}. In section \ref{sec:linear homogeneous symmetries} and \ref{sec:nonlinear homogeneous symmetries} we discuss how to find such flows for a large class of functions $H(x)$.

\begin{remark}[Pseudo-invariant-preserving methods]\label{rem:pseudo-energy}
   If such a $\psi_t$ is not known for a particular $H(x)$, it can be approximated to order $\mathcal{O}(h^{q+1})$ using an order-$q$ one-step method $\Psi_t\approx\psi_t$. This yields pseudo-invariant-preserving methods. This is formalised for the general case is shown in theorem \ref{thm:pseudo-energy} in section \ref{sec:nonlinear homogeneous symmetries}.
\end{remark}

Notice that the method can be easily generalised to ODEs that modulate the invariant at a specified rate. For example, linear dissipation, where if $\dot{H}(x)=-\alpha H(x)$, for some positive $\alpha\in\R$ then $H(\varphi_h(x))=e^{-\alpha h}H(x)$ (actually in this case $H(x)$ is also homogeneous with respect to $\varphi_t$). So to construct a homogeneous projection method that modulates invariants, one can simply project onto the surface satisfying $H(\widehat{\Phi}_h(x))=H(\varphi_h(x))$. This is summarised in the following corollary.

\begin{corollary}[Invariant-dissipating methods]
  Let $H(x)$ be homogeneous with respect to $\psi_t$, $\Phi_h(x)$ an order-$p$ method and $ H(\varphi_h(x))=\alpha(x, h)$ for some known $\alpha(x, h)$. Then the method
  $$
    \widehat\Phi_{h}:=\psi_{s}\circ\Phi_h(x), \quad s:= \frac{1}{k}\log\left(\frac{\alpha(x, h)}{H(\Phi_h(x))}\right)
  $$
  satisfies $H(\widehat\Phi_{h}(x)) = \alpha(x, h)$, and is order-$p$.
\end{corollary}
\begin{proof}
  The proof is identical to theorem \ref{thm:hom proj method}, except that we replace $H(x)$ with $\alpha(x, h)$ in the definition of $s$.
\end{proof}
Though it is not obvious from the defintion, the homogeneous projection method is a special case of the usual projection method.
\begin{proposition}[Equivalence with standard projection methods]
  Let $\widehat{\Phi}_h(x)$ be a homogeneous projection method as per definition \ref{def:hom proj method}, then $\widehat{\Phi}_h(x)$ is a projection method according to definition \ref{def:projection method}. 
\end{proposition}
\begin{proof}
  As $\psi_t(x)$ is the flow of $g(x)$, we can write $\psi_t(x) = x + \int_0^t g(\psi_s(x)) ds.$ Then letting $v(x) = \int_0^t g(\psi_s(\Phi_h(x))) ds$ and $\lambda=1$, $\widehat{\Phi}_h(x)$ is of the form \eqref{eq:proj method}. 
\end{proof}
This is especially clear when $\psi_t(x)$ is linear, in which case $\psi_t(x) = e^{tA}x$ for some matrix $A\in\R^{n\times n}$ and $g(x) = Ax$. Then the homogeneous projection method can be written as
$$\widehat{\Phi}_h(x) = e^{sA}\Phi_h(x) = \Phi_h(x) + (e^{sA}-I)\Phi_h(x),$$
which is of the form \eqref{eq:proj method} with $\lambda=1$, $v(x) = (e^{sA}-I)\Phi_h(x)$ and satisfies $H(\widehat{\Phi}_h(x)) = H(x)$. In other words, the homogeneous projection method is a standard projection method expressed using the flow of a vector field $g(x)$.

As is the case with all projection methods, the projection operator does not preserve other geometric properties of the base method $\Phi_h(x)$, such as symplecticity, time-reversal symmetry, volume-preservation, affine equivariance and so forth. However, the order of accuracy is retained as shown in theorem \ref{thm:hom proj method}. Due to this, it is natural to explore their use in error controlled methods, where geometric properties are notoriously difficult to enforce.

\subsection{Multiple homogeneous invariants}
Often ODEs possess multiple invariants, say $(H_1(x),...,H_m(x))^T\in\R^m$. If each $H_i(x)$ is homogeneous with respect to $e^{s_jA_j}$ with weight $k_{ij}$, i.e., $H_i(e^{s_j A_j} x) = e^{k_{ij} s_j} H_i(x)$, then it is natural to ask if one can preserve these invariants simultaneously by the linear flow $\psi_s = e^{s_1 A_1}...e^{s_m A_m}$. In other words, given the $m$ generators $A_i$, find the $m$ parameters $s = (s_1,...,s_m)\in\R^m$ such that $(H_1(x),...,H_m(x))^T$ is homogeneous with respect to $\psi_s$. We now show that this is sometimes possible but requires a strict set of conditions to be met. 
\begin{lemma}
   Let $A_i\in\R^{n\times n}$ be a set of commuting generators for $i=1,...,m$, where $[A_i,A_j] = 0$, and $H_i(x)$, $\psi_s$ and $k_{ij}$ as given above. Then, if the degree matrix $K=[k_{ij}]$ is invertible, the homogeneous projection method $\widehat{\Phi}_h = \psi_s\circ\Phi_h$, preserves $H(x)$, with $s$ being the solution to the linear system
\begin{equation}\label{eq:Kstar-system}
s = K^{-1}b,
\end{equation}
for $b_i =\log\!\frac{H_i(x)}{H_i(\Phi_h(x))}$.
\end{lemma}
\begin{proof}
Let $\tilde{x}=\Phi_h(x)$. We want to be able to compute the parameters $s_j$ for $j=1,...,m$ such that $\psi_t$ projects each $H_i(\tilde x)$ onto the invariant preserving manifold. That is $H_i(\psi_s(\tilde{x}))=H_i(x)$ for each $i$. Then under the linear group action $\psi_s = e^{s_1 A_1}...e^{s_m A_m}$ we have
\begin{equation}
    H_i(\psi_s(\tilde{x})) = H_i(\prod_{j}^m e^{s_j A_j} \tilde{x}) = \exp\!\Big(\sum_{j=1}^m k_{ij}\,s_j\Big)\,H_i(\tilde{x}).
\end{equation}
So $s_j$ must be chosen so that the above equals $H_i(x)$ for each $i$. Assuming $H_i(x)H_i(\tilde{x})>0$, setting the right hand side equal to $H_i(x)$ then dividing the above by $H_i(\tilde{x})$ and taking the logarithm yields $Kt = b$. If $K$ is nonsingular, the unique solution is $s=K^{-1}b$.
\end{proof}

One example is for functions $H_i$ that are homogeneous with respect to $\psi^{[i]}_t$, but invariant with respect to $\psi^{[j]}_t$ for $j\neq i$. This arises for functions that depend on separate variables, e.g., if $x=(x_1, x_2)\in\R^{n_1+n_2}$ for $x_1\in\R^n_1$, $x_2\in\R^n_2$ and one has two homogeneous invariants $H_1(x_1)$ and $H_2(x_2)$ with respect to $\psi_t^{[1]}:x \mapsto (e^{tA_1}x_1, x_2)$ and $\psi_t^{[2]}:x \mapsto (x_1, e^{tA_2}x_2)$, respectively, then the degree matrix $K$ can be made diagonal and both invariants can be simultaneously preserved. For non-diagonal invertible $K$, one requires that the invariants are homogeneous with respect to multiple commuting actions forming invertible degree matrix. If these conditions are not met, one can still apply alternating homogeneous projections as per definition \ref{def:alt proj} or resort to the pseudo-invariant-preserving methods discussed in the section \ref{sec:nonlinear homogeneous symmetries}.

It is more common for a system to possess several invariants $H_i\in C^{1}(U)$, for $i=1,\ldots,m>1$, each homogeneous with respect to its own flow $\psi^{H_i}_t$. When these actions are distinct and non-commuting, a single simultaneous projection is generally unavailable. In such cases, one can apply a sequence of single-invariant projections in a cyclic fashion to obtain approximate simultaneous preservation of all invariants over multiple steps.

\begin{definition}[Alternating homogeneous projection]\label{def:alt proj}
Let $\{H_i\}_{i=1}^m$ be functions that are homogeneous with respect to the flows $\psi^{H_i}_t$. Let $\Phi_h$ be a base one-step method and let $\widehat{\Phi}_h^{H_i}=\psi^{H_i}_{s_i}\!\circ\!\Phi_h$ denote the homogeneous projection of $\Phi_h$ associated with $H_i$ as per definition~\ref{def:hom proj method}. Then the alternating homogeneous projection method is the $m$-step composition
\begin{equation}
  \widehat{\Phi}_{mh}(x)
  =\widehat{\Phi}_h^{H_m}\circ\widehat{\Phi}_h^{H_{m-1}}\circ\cdots\circ\widehat{\Phi}_h^{H_1}(x),
\end{equation}
where each stage applies the projection corresponding to one invariant.
\end{definition}

That is, the base integrator is followed by sequential homogeneous projections, each restoring one invariant at a time.  
While a single step preserves only one $H_i$ exactly, all invariants are restored cyclically, each at least once every $m$ steps. Even tho we haven't proved it here, we observe bounded growth in the errors of all invariants over long time intervals in our numerical experiments. Alternating projection also comes with the advantage of being fast to implement, requiring only an additional cheap evaluation of one $\psi^{H_i}_t$ for each step. 

So far, the methods rely on the existence of a homogeneous symmetries $\psi_t$. In the following section we will show that a large class of functions have \textit{linear} homogeneous symmetries that can be computed in closed form with matrix exponentials. Then in the subsequent section we will show that almost all functions have a \textit{non-linear} homogeneous symmetry that can be efficiently approximated to any desired level of accuracy.

\section{Linear homogeneous symmetries}\label{sec:linear homogeneous symmetries}
The set of functions $H(x)$ and flows $\psi_t(x)$ that satisfy the homogeneity condition \eqref{hom cond} will now be discussed. Here, we consider the case where $\psi_t(x)$ is linear. The non-linear case is discussed in the next section. The first situation leverages a more common notion of homogeneity, where $\psi_t(x)$ is the flow of a linear vector field, i.e., $\frac{d}{dt}\psi_t(x)=A\psi_t(x)$ for some matrix $A\in\R^{n\times n}$. This situation is useful because we can achieve energy preservation exactly with one cheap operation, however such a flow only exists if $H(x)$ has a special, though rather general, structure. The second situation is more general and allows for non-linear group actions, where $\psi_t(x)$ is the flow of a non-linear vector field, i.e., $\frac{d}{dt}\psi_t(x)=g(\psi_t(x))$ for $g:\R^n\to\R^n$. This will be discussed in the next section.

Consider the case where $H(x)$ is homogeneous with respect to an affine transformation $$\psi_t(x)=e^{tA}x + b,$$ for some matrix $A\in\R^{n\times n}$ and $b\in\R^n$. The case $b=0$ corresponds to homogeneity with respect to the origin, and the $b\ne 0$ case corresponds to homogeneity about another point. The translation component $b$ can be removed by a change of variables, so without loss of generality we will set $b=0$ for simplicity. 

\subsection{Isotropic dilations.} The simplest case, where $\psi_t(x)=e^t x$, representing uniform scaling in the direction of $x$. The homogeneous condition is $H(\lambda x)=\lambda^k H(x)$, where we have set $\lambda = e^{t}$. Here, the infinitesimal condition reduces to $x^T\nabla H(x)=k H(x)$, which is the standard definition of Euler's equation \eqref{euler eq} for homogeneous functions. Examples of homogeneous functions under isotropic dilations include: 
\begin{enumerate}
\item Star shaped functions $H(x) = f(\|x\|)K(x/\|x\|)$ where $f$ is a homogeneous function.
\item Perspective projections $H(x) =  f(x_n) K(\frac{x_1}{x_n},..., \frac{x_{n-1}}{x_n})$.
\item Homogeneous polynomials such as pure quadratics $H(x) = x^T A x$. 
\item Rational homogeneous polynomials $H(x) = \frac{P(x)}{Q(x)}$ where $P,Q$ are homogeneous polynomials.
\end{enumerate}
\subsection{Weighted dilations.} A more general case where $\psi_t(x)=e^{tD}x$ for some diagonal matrix $D=\text{diag}(d_1,\ldots,d_n)$, where $d_i\in\R$. This represents different scaling rates in different directions according to the values of $d_i$, allowing for a notion of homogeneity for a broader class of functions. Some examples are:
\begin{enumerate}
\item Mechanical systems with homogeneous potentials $H(q,p) = \frac{1}{2} p^\top M^{-1} p + V(q)$ where $V(\lambda q) = \lambda^r V(q)$ for some $r\in\R$, admits the linear group action $\psi_t(q,p) = (e^tq, e^{\frac{rt}{2}}p)$ with $k=r$. 
\item Geodesic flow on a Riemannian manifold with $H(q,p) = \frac{1}{2} p^\top G(q)^{-1} p$ where $G(q)$ is a positive definite matrix admit the linear group action $\psi_t(q,p) = (q, e^{t}p)$ with $k=2$.
\item Weighted homogeneous polynomials. Let $x^\alpha = x_1^{\alpha_1} x_2^{\alpha_2} \cdots x_n^{\alpha_n}$ be a monomial with multi-index $\alpha = (\alpha_1,\ldots,\alpha_n)$. Then the polynomial $H(x) = \sum_{\alpha} c_\alpha x^\alpha$ is weighted-homogeneous of degree $k$ with respect to weights $w=(w_1,\ldots,w_n)$ if each monomial satisfies $\alpha\cdot w = k$. The corresponding linear group action is $\psi_t(x) = (e^{w_1 t} x_1, \ldots, e^{w_n t} x_n)$.
\end{enumerate}
For example $n$-body problems of the form 
\begin{equation}
    H(q,p) = \sum_{i=1}^n \frac{p_i^2}{2m_i} + \sum_{1\leq i<j\leq n} \frac{m_im_j}{\|q_i-q_j\|}
\end{equation}
admit the linear group action $\psi_t(q,p) = (e^{-2t}q, e^{t}p)$ with $k=2$. 

\subsection{Matrix actions.} 
A further case arises when $H$ is a function on matrices, $H:\R^{n\times n}\to\R$, and the group action acts by left--right multiplication or conjugation. For instance, $\psi_t(X) = e^{tA} X e^{tB}$, where $A,B \in \R^{n\times n}.$ The homogeneous condition is $H(e^{tA} X e^{tB}) = e^{k t} H(X)$ with infinitesimal form $\langle AX + X B ,\,\nabla_X H(X) \rangle_F \;=\; k\,H(X),$ where $\langle\cdot,\cdot\rangle_F$ denotes the Frobenius inner product. 

A canonical example is the determinant $H(X)=\det X$. One computes $ \det(e^{tA} X e^{tB}) = e^{(\tr A + \tr B)t}\,\det X,$ so $H$ is a homogeneous function of degree $k=\tr A+\tr B$ under this action. Thus $\det$ is a \emph{relative invariant} (or \emph{characteristic}) of the matrix group action $\psi_t$.

\subsection{Quadratic invariants.}
Quadratic forms are always homogeneous with respect to isotropic dilations. However, they can sometimes be homogeneous with respect to a broader class of linear actions. Let $H(x) = x^\top S x$ be quadratic with $S=S^\top \in \mathbb{R}^{n\times n}$. Under the linear action $\psi_t(x) = e^{tA}x$, solving the infinitesimal condition for $A$ reduces to solving a special case of Sylvester's equation $A^\top S + SA \;=\; \kappa\,S.$

\subsection{Rank-deficient linear actions.}
Let $A\in\R^{n\times n}$ and $w_i\in\R^{n}$ be such that $A^Tw_i=0$ for $i=1,\ldots,\ell$ and $\ell\ge 1$. In this case, the infinitesimal condition simplifies to $$(Ax)^T\nabla H(x)=k H(x),$$ which is a first order linear PDE for $H(x)$ and can be solved by the straightening theorem \cite{liveranidynamical}
$$H(x)=\exp\!\big(k\,s(x)\big)\,G\!\big(w_1^Tx,\dots,w_{\ell}^Tx\big),$$
for some $s(x)$ satisfying $\dot{s}=1$ and for some $G:\R^{\ell}\to\R$. Then $H:\R^n\to\R$ is homogeneous of degree $k$ with respect to $\psi_t(x)=e^{tA}x$.
\subsection{Homogeneous projection and extended phase space methods using auxiliary variables.}
Homogeneous projection methods can be combined with extended phase space methods that introduce auxiliary variables to reformulate complicated invariants into homogeneous ones. For example, consider a function $K(x)\in C^1(U)$. By introducing the auxiliary variable $y\in\R$ we can transform $K(x)$ into the homogeneous function $H(x,y) = y^k K(x/y)$ for some integer $k>0$. Then $H(x,y)$ is always isotropically homogeneous of degree $k$. This mirrors the idea of invariant energy quadratisation and scalar auxiliary variable methods \cite{shen2018scalar, zhang2020novel}, but here we can achieve exact invariant-preservation in the extended phase space using homogeneous projection methods.

\section{Non-linear homogeneous symmetries}\label{sec:nonlinear homogeneous symmetries}
If $H(x)$ doesn't admit any homogeneous symmetries with respect to a linear group action, we can instead look for non-linear actions. We consider two cases, the first being conjugate homogeneity, where a non-homogeneous function can be transformed into a homogeneous one via a diffeomorphism. This can lead to closed form non-linear group actions $\psi_h$. The second case is the most general, where we look for a non-linear group action that satisfies the homogeneity condition directly, by defining it through its infinitesimal generator. As we will see, there always exists such a $g(x)$ for any set of functions $H_i(x)$, and therefore $\psi_t$ exists and can be approximated accurately by any numerical method.

\subsection{Invariants that are conjugate to homogeneous}
We now introduce a broader class of homogeneous functions by allowing conjugation with a diffeomorphism. 

\begin{proposition}[Conjugate homogeneity]\label{prop:conj hom}
Let $\phi:\R^n\to\R^n$ be a diffeomorphism and let $G$ be homogeneous of degree $k$ with respect to $e^{tA}$ for $A\in\R^{n\times n}$, then $H(x) = G(\phi(x))$ is homogeneous of degree $k$ with respect to  
$$ 
    \psi_t(x) = \phi^{-1}\!\circ e^{At}\circ\phi(x). 
$$
\end{proposition}
Thus conjugate homogeneity provides a systematic way to generate large families of nonlinear homogeneous functions by composing simpler homogeneous forms with expressive diffeomorphisms with closed form inverse. For example, many splitting methods have closed form inverses such as shears, symplectic splittings \cite{mclachlan2004explicit} or the following. 

\begin{definition}[Triangular diffeomorphism] 
Let $x_{<i} = (x_1,...,x_{i-1})$, then a diffeomorphism $\phi:\R^n\to\R^n$ is called \emph{triangular} if
$$\phi_i(x_1,\dots,x_n) = a_i(x_{<i})h_i(x_i) + b_i(x_{<i}), \qquad i=1,\dots,n.$$
It's inverse is given by backsubstitutions of the form 
$$x_i = h_i^{-1}\!\left(\frac{\phi_i - b_i(x_{<i})}{a_i(x_{<i})}\right), \qquad i=1,\dots,n.$$
\end{definition}
Such maps are always have closed form inverse provided that $h_i^{-1}$ is available in closed form. The inverse $\phi^{-1}$ can be computed recursively by solving the equations $\phi_i(x_1,\dots,x_i)=y_i$ for $x_i$. Triangular diffeomorphisms therefore provide a simple yet expressive family of nonlinear invertible transformations that preserve computational tractability. Any finite composition $\phi=\phi_m\circ\dots\circ\phi_1$ of such building blocks yields a diffeomorphism with explicit inverse, thereby producing large families of conjugate homogeneous functions. 

This can be especially useful when working with invariants of the mechanical form $H = p^TM(q)p + V(q)$, where there always exists a triangular diffeomorpshism that transforms the kinetic term into a homogeneous function. 

\begin{proposition}
  Let  $H(q, p) = \tfrac{1}{2}p^TM(q)p + V(q)$, where $V(q)$ be homogeneous of degree $k$ with respect to $\psi^V_t$ and $M(q)$ be a positive definite matrix with Cholesky decomposition $M(q)=L(q)L(q)^T$ for lower triangular matrix $L(q)$. Denote by $\phi(q,p) = (q, v)$, a triangular diffeomorphism with $v=L(q)p$. Then $H\circ\phi$ is homogeneous of degree $k$ with respect to $\psi_t(q, v) = (\psi^{V}_t (q), e^{\tfrac{1}{2}kt}v)$.
\end{proposition}
So by the above proposition, if $V(q)$ is already homogeneous, or can be made homogeneous using another triangular diffeomorphism, then we can always find energy preserving maps. This is equivalent to finding a new set of conjugate momenta where the mass matrix is diagonal. The following example summarizes the two degree of freedom case.
 
\begin{example}[Mechanical system with two degrees of freedom]\label{eg:2d mechanical triangular map} Let $q, p\in\R^2$ and 
  $$M(q)=\begin{pmatrix}a(q)&b(q)\\[2pt]b(q)&c(q)\end{pmatrix},\qquad
\Delta(q)=a(q)c(q)-b(q)^2 > 0.$$
with the Cholesky factorisation
$$M(q)=L(q) L(q)^T,\qquad
L(q)=\begin{pmatrix}
\lambda_{11}(q) & 0\\[4pt]
\lambda_{21}(q) & \lambda_{22}(q)
\end{pmatrix},$$
and 
$$\lambda_{22}(q)=\sqrt{c(q)},\quad
\lambda_{21}(q)=\frac{b(q)}{\sqrt{c(q)}},\quad
\lambda_{11}(q)=\sqrt{a(q)-\frac{b(q)^2}{c(q)}}=\sqrt{\frac{\Delta(q)}{c(q)}}.$$
Define the triangular diffeomorphism $\phi:(q,p)\mapsto (q, y)$, where $y=L(q)p$, i.e., 
$$y_1 = \lambda_{11}(q)p_1,\quad y_2 = \lambda_{21}(q)p_1+\lambda_{22}(q)p_2$$

Then, identically, $p^{T}M(q)p = \|L(q)p\|^2 = y_1^2+y_2^2$. The inverse is given by
$$p_1=\frac{y_1}{\lambda_{11}(q)},\qquad
p_2=\frac{y_2-\lambda_{21}(q)p_1}{\lambda_{22}(q)}.
$$
\end{example}

We will demonstrate this method on the double pendulum in the numerical experiments section.

\subsection{Explicit pseudo-invariant-preserving homogeneous projection methods}
Here we present the most general case, where $f(x)$ has multiple invariants $H_i$ for $i=1,...,m$ with no known linear homogeneous symmetries. Let $\psi_t$ be the flow of the vector field $g(x)\in C^{q+1}(U)$, i.e., $\frac{d}{dt}\psi_t(x)=g(\psi_t(x))$. The infinitesimal condition becomes $g(x)^T\nabla H_i(x)=k_i H_i(x)$, for each $i$. Let $G(x) = (\nabla H_1(x), \ldots, \nabla H_m(x))\in\R^{n\times m}$ be the gradient matrix, and $V(x) = (v_1(x), \ldots, v_m(x))\in\R^{n\times m}$ be matrix, then a solution on compact sets where $\det{G^TV}\ne 0$ is 
\begin{equation}\label{eq:non lin generator}
g(x)=V(G^TV)^{-1}K\mathbf{1},
\end{equation}
where $K = \diag(k_1 H_1, \ldots, k_m H_m)$ is a diagonal matrix in $\R^{m\times m}$. The choice $V = G$ yields the minimum-norm vector field satisfying the infinitesimal condition, also known as orthogonal projection, which is usually a safe choice. Otherwise Tikhonov regularization of $G^TV$ can be applied for poorly conditioned systems. As the flow of \eqref{eq:non lin generator} is generally not available in closed form, we approximate it with a one-step method $\Psi_t$ of order $q$. Using this approximation in the homogeneous projection method yields the following method.
 
\begin{theorem}[Pseudo-invariant-preservation]\label{thm:pseudo-energy}
  Let $\Phi_h$ be an order-$p$ one-step method for $\dot x=f(x)$ with exact flow $\varphi_h$ with $m$ invariants $f^T\nabla H_i=0$, such that $H_i(x)H_i(\Phi_h(x))>0$ and each $H_i$ are Lipschitz continuous. Furthermore, let $\Psi_h$ be an order-$q$ one-step method for $\dot z=g(z)$, as defined. Then the method
  $$ \widehat{\Phi}_h(x)\;:=\;\Psi_{1}\circ\Phi_h(x), $$
  where $k_i \;=\; \log\!\left(\frac{H_i(x)}{H_i(\Phi_h(x))}\right)$ for $i=1,...,m$ satisfies the following:
\begin{enumerate}
  \item $|H_i(\widehat{\Phi}_h(x)) - H_i(x)| = \mathcal{O}\big(h^{(p+1)(q+1)}\big)$ for each $i$ (pseudo-invariant-preserving), and 
  \item  $|\widehat{\Phi}_h(x) - \varphi_h(x)\| = \mathcal{O}(h^{p+1})$ (same order as $\Phi_h$).
\end{enumerate}

\end{theorem}

\begin{proof} 

Define for each $i$ the energy error $\delta_{i,h}:=H_i(\Phi_h(x))-H_i(x)$, which satisfies
$$ |\delta_{i,h}|\le L_{H_i}\,\|\Phi_h(x)-\varphi_h(x)\|=\mathcal O(h^{p+1}).$$
Letting $g$ be defined as above gives $g(z)^\top\nabla H_i(z)=k_i H_i(z)$, by construction, hence along $\psi_t$ we have $H_i(\psi_t(z))=e^{k_i t}H_i(z)$. Evaluating at $z=\Phi_h(x)$, $t=1$ and $k_i \;=\; \log\!\left(\frac{H_i(x)}{H_i(\Phi_h(x))}\right)$ gives 
\begin{equation}\label{eq:pseudo energy exact}
  H_i\big(\psi_1(\Phi_h(x))\big)=e^{k_i}H_i(\Phi_h(x))=H_i(x).
\end{equation}
The order of $\Psi_h$ implies that
$$
\|\Psi_1(z)-\psi_1(z)\|\;=\;\mathcal O\!\big(\|g\|_{C^q}^{\,q+1}\big).
$$
Since $g_i$ is bounded and depends linearly on $k_i$, we have $\|g\|_{C^q}=\mathcal O(|k|)$, hence
$$
\|\Psi_1(z)-\psi_1(z)\|=\mathcal O(|k|^{\,q+1}).
$$
Evaluating this at $z=\Phi_h(x)$ and applying the mean value inequality for $H_i$ gives
$$
\big|H_i(\Psi_1(\Phi_h(x)))-H_i(\psi_1(\Phi_h(x)))\big|
\le L_{H_i}\,\|\Psi_1(\Phi_h(x))-\psi_1(\Phi_h(x))\|
=\mathcal O(|k|^{\,q+1}).
$$
Combining with \eqref{eq:pseudo energy exact} and $|k|=\mathcal O(h^{p+1})$ yields, for each $i$,
$$
\big|H_i(\Psi_1(\Phi_h(x)))-H_i(x)\big|
=\mathcal O\!\big(h^{(p+1)(q+1)}\big),
$$
which proves item 1. For the state error, use the triangle inequality
$$
\|\Psi_1(\Phi_h(x))-\varphi_h(x)\|
\le \|\Phi_h(x)-\varphi_h(x)\|+\|\psi_1(\Phi_h(x))-\Phi_h(x)\|+\|\Psi_1(\Phi_h(x))-\psi_1(\Phi_h(x))\|.
$$
The first term is $\mathcal O(h^{p+1})$ by the order of $\Phi_h$. For the second term,
$$
\|\psi_1(\Phi_h(x))-\Phi_h(x)\|
\le \int_0^1 \|g(\psi_t(\Phi_h(x)))\|\,dt \le C\,|k|=\mathcal O(h^{p+1}),
$$
for some constant $C$. The final term is $\mathcal O(|k|^{\,q+1})=\mathcal O(h^{(p+1)(q+1)})$. Therefore
$$
\|\Psi_1(\Phi_h(x))-\varphi_h(x)\|=\mathcal O(h^{p+1}),
$$
which proves item 2.
\end{proof}

This method can be seen as a generalization of the linear homogeneous projection method to the non-linear case, where the linear group action is replaced by a non-linear one defined through its infinitesimal generator. The method is explicit and easy to implement, requiring only the evaluation of the invariants and their gradients. The choice of $V$ in \eqref{eq:non lin generator} can be used to control properties of the generator $g$. Furthermore, the method can be iterated multiple times to achieve superlinear convergence to invariant preservation.

\begin{corollary}[superlinear iterated convergence]\label{cor:superlinear}
Let $\widehat\Phi_h=(\Psi_1)^r\circ\Phi_h$ denote the pseudo invariant-preserving method with $r$ iterations of the $\Psi_1$, where $k_i$ are updated each iteration. Then
$$
\big|H_i(\widehat\Phi_h)-H_i(x)\big|= O\big(h^{(p+1)(q+1)^r}\big),
$$
for $i=1,...,m$. 
\end{corollary}
\begin{proof}  
  After $r$ iterations, define the invariant error as $ \delta^{(r)}_{i,h}=H_i(\left(\Psi_1\right)^r\circ\Phi_h(x))-H(\varphi_h(x))$. Theorem \ref{thm:pseudo-energy} gives $|\delta^{(1)}_{i,h}|=\mathcal O(h^{(p+1)(q+1)})$ when the $r=0$ (base step invariant error) is $|\delta_{i,h}|=|\delta^{(0)}_{i,h}|=\mathcal O(h^{(p+1)})$. The result for $|\delta^{(r)}_{i,h}|$ for $r>1$ then follows by induction, by noting that each iteration of $\Psi_1$ increases the order of invariant preservation from $(p+1)(q+1)^{r-1}$ to $(p+1)(q+1)^r$.
\end{proof} 

The convergence bound $O\big(h^{(p+1)(q+1)^r}\big)$ is verified numerically in appendix \ref{app:energy convergence}. 
  
\section{Numerical experiments: ODEs with invariants}\label{sec:numerical experiments ODEs}

\subsection{Invariant-preserving error controlled methods using \texttt{scipy} implementation.} 
Adaptive methods are widely used in practice, but it is non-trivial to combine adaptivity with structure-preservation. However, projection methods like ours are well suited to open source implementations with minimal modifications. To this end, for the base step $\Phi_h(x)$, we consider popular methods such as an explicit Runge-Kutta method of order 5(4) using the Dormand-Prince pair formulas \cite{dormand1980family}. The error is controlled assuming accuracy of the fourth-order method accuracy, but steps are taken using the fifth-order accurate formula. This is usually referred to as the RK45 method. We also consider the Dormand-Prince embedded 8(5,3) method (DOP853) \cite{hairer1987solving}, which is a popular explicit Runge–Kutta method with embedded error estimates. The logic of the adaptive methods remain unchanged. We only apply a projection step after the step is accepted. Our implementation follows the \texttt{scipy.integrate.solve\_ivp} interface \cite{2020SciPy-NMeth}, the only difference being a thin wrapper on the step function to apply the projection after each accepted step. The code is available at the authors github repository\footnote{github.com/bentaps/homproj} or via \texttt{pip install homproj}.

\subsection{Double pendulums}
The purpose of this experiment is to compare the accuracy of energy-preserving methods with symplectic methods. We consider the double pendulum systems because they have one preserved invariant (the Hamiltonian) and a symplectic structure. As the Hamiltonian is non-separable, there do not exist fast symplectic methods to solve these systems. In such situations, whether preserving energy or symplectic structure is best is not clear. Here, we will demonstrate that adaptive methods are advantageous over fixed step size methods. 

The double pendulum system which describes the motion of two masses connected by a rod of fixed length. Let $x=(q_1, q_2, p_1, p_2)$, then the Hamiltonian is given by 
\[
H(x)=
\frac{p_{1}^{2}+2p_{2}^{2}-2\cos(q_{1}-q_{2})\,p_{1}p_{2}}
{2\left(2-\cos^{2}(q_{1}-q_{2})\right)}
+V(q_{1}, q_{2}).
\]

Following example \ref{eg:2d mechanical triangular map}, we consider the triangular map $\phi:\R^4\to\R^4$ given by
\begin{equation}\label{eq:dp diffeo}
  \phi(x)=y=\begin{pmatrix}
h_1(q_1)\\[2pt]
h_2(q_2)\\[2pt]
\tfrac{1}{2}\,p_{1}\\[6pt]
\dfrac{2p_{2}-\cos(q_{1}-q_{2})\,p_{1}}{2\sqrt{2-\cos^{2}(q_{1}-q_{2})}}
\end{pmatrix},
\end{equation}
for some invertible functions $h_1, h_2$ described below. Two double pendulum potentials are considered here:
\begin{enumerate}
  \item A torsioned joint potential: $V_1(q_1,q_2)=-\frac{1}{2}(q_1^2+q_2^2)$. This is a double pendulum where the joints are acted upon by torsion springs, i.e., $F_i=-k_i q_i$ for some spring constants $k_i$. For $h_1(q_1)=q_1$, $h_2(q_2)=q_2$, we have $H(x)=G_2(\phi(x))$ where
  $$G_2(y)=\frac{1}{2}(y_1^2+y_2^2)+y_3^2+y_4^2,$$
  is homogeneous of degree $1$ with respect to the isotropic scaling.
  \item A gravitational potential: $V_2(q_1,q_2)=-2\cos(q_1)-\cos(q_2)$. This is the usual double pendulum where the masses are acted upon by gravity. Then for $h_1(q_1)=-2\cos(q_1)$, $h_2(q_2)=-\cos(q_2)$, we have $H(x)=G_1(\phi(x))$ where
  $$G_1(y)=y_1+y_2+y_3^2+y_4^2,$$
  is homogeneous of degree $1$ with respect to $e^{tA}$ for $A=\operatorname{diag}(1,1,\tfrac{1}{2},\tfrac{1}{2})$.
\end{enumerate}
So $H(x)$ is conjugate homogeneous of degree $1$ with respect to $\psi_t=\phi^{-1}\circ e^{tA}\circ\phi$ according to proposition \ref{prop:conj hom}. The conjugate homogeneous projection method defined by definition \ref{def:hom proj method} is an explicit energy preserving method. 

\subsubsection{Results}
We will perform three experiments: 
\begin{enumerate}
  \item A comparison of fixed-step size methods for the torsioned joint potential. Here the purpose is to explore the effect of energy preservation on accuracy when it comes to fixed step size methods, and compare to symplectic methods. 
  \item A comparison of fixed-step size methods for the gravitational potential. Here their purpose is to measure the accuracy and computational cost of each method. Here, we show an example of the Conjugate Homogeneous (CH) projection method performing poorly due to conditioning issues, but that the Pseudo Nonlinear Homogeneous (PNH) Projection still works well.  
  \item A comparison of adaptive time-stepping methods for the torsioned joint potential. Here the purpose is to show that adaptive time-stepping methods can outperform fixed step size symplectic methods in terms of accuracy and computational cost, and is easily implemented using our projection framework. 
\end{enumerate}

\paragraph{Description of methods} In all the experiments we will consider two explicit energy-preserving homogeneous projections of the form \ref{def:hom proj method}. The first projection is the conjugate homogeneous method (CH) using proposition \ref{prop:conj hom} and \ref{eq:dp diffeo} for the $\psi_t$. The second is a pseudo-nonlinear homogeneous projection (PNH) using theorem \ref{thm:pseudo-energy}, where for the one-step method $\Psi_h\approx \psi_t$, we use a second-order Runge-Kutta method so that the Hamiltonian is preserved up to order $O(h^{3(p+1)})$. The base step will be a fourth order Runge-Kutta method (RK4) for experiments 1 and 2, and the Dormand-Prince (RK45) method for experiment 3.

We also consider the following methods for comparison: (1) a RK4 with no projection; (2) an explicit pseudo-symplectic method of order 4 and symplectic order 9 (PS4(9)) \cite{mclachlan2024runge, tao2016explicit}; and (3) a 4th-order symplectic Gauss collocation Runge–Kutta (GLL4). For experiment 3, we will compare against the adaptive RK45 and a 6th-order symplectic Gauss collocation Runge–Kutta method (GLL6).

\paragraph{Experiment 1: fixed-step methods for the torsioned joint potential} 
We generate 10 initial conditions of the form $(q_1, q_2, p_1, p_2) = (\delta_1, \delta_2, 1, -1)$, where $|\delta_i|\le \tfrac{1}{10}$ are randomly chosen. We integrate over the time interval $t\in[0,500]$ with step-size $h=0.05$. The results are shown in figure \ref{fig:dp-errors}. 
We note that the CH and PNH projection methods significantly improve the error over the base RK4 method, for negligible computational overhead. The CH method, notably, only adds a 3\% runtime to achieve machine-precision energy preservation, while the PNH method adds about 40\% runtime. The two homogeneous projection methods have comparable accuracy as the expensive GLL4 method, which is over 3-4 times slower.

\paragraph{Experiment 2: fixed-step methods for the gravitational potential} 
Here, the methods and setup is exactly the same as experiment 1, except we use the gravitational potential $V_2$ and we integrate over the time interval $t\in[0,100]$ with step-size $h=0.01$. In this case, we see that the CH projection method is actually the least accurate. The PNH method is significantly more accurate than all the other methods, even though energy isn't preserved to machine precision this time. 

We do note that the CH projection performs worst here. A possible explanation is due to the conditioning of the conjugation map, when $h_i$ in equation \eqref{eq:dp diffeo} are trigonometric functions, as the derivative of $\cos^{-1}$ blows up near the boundary of its domain. This can lead to error amplification. Letting $J(x) = D\phi(x)$ denote the Jacobian matrix of the conjugation map, and $\psi_h = \phi^{-1} \circ e^{hA} \circ \phi$ being the homogeneous projection map then 
$$
D\psi_h(x) = \big(J(\psi_h(x))\big)^{-1} e^{hA} J(x) 
  = I + h\, J(x)^{-1} A J(x) + O(h^2).
$$
Hence, if $J(x)$ is poorly conditioned, such is the case when it contains terms like $\cos^{-1}$, then $\|J(x)^{-1}\|$ can dominate the $O(h)$ term above and therefore degrade the accuracy. Therefore, some care needs to be taken to ensure that the conjugation map is well-conditioned in the region of interest. We would therefore recommend checking the conditioning of the map at each time step and swapping to the PNH projection when necessary.

\begin{figure}[ht] 
  \centering
  \begin{subfigure}{0.8\textwidth}
    \includegraphics[width=\textwidth]{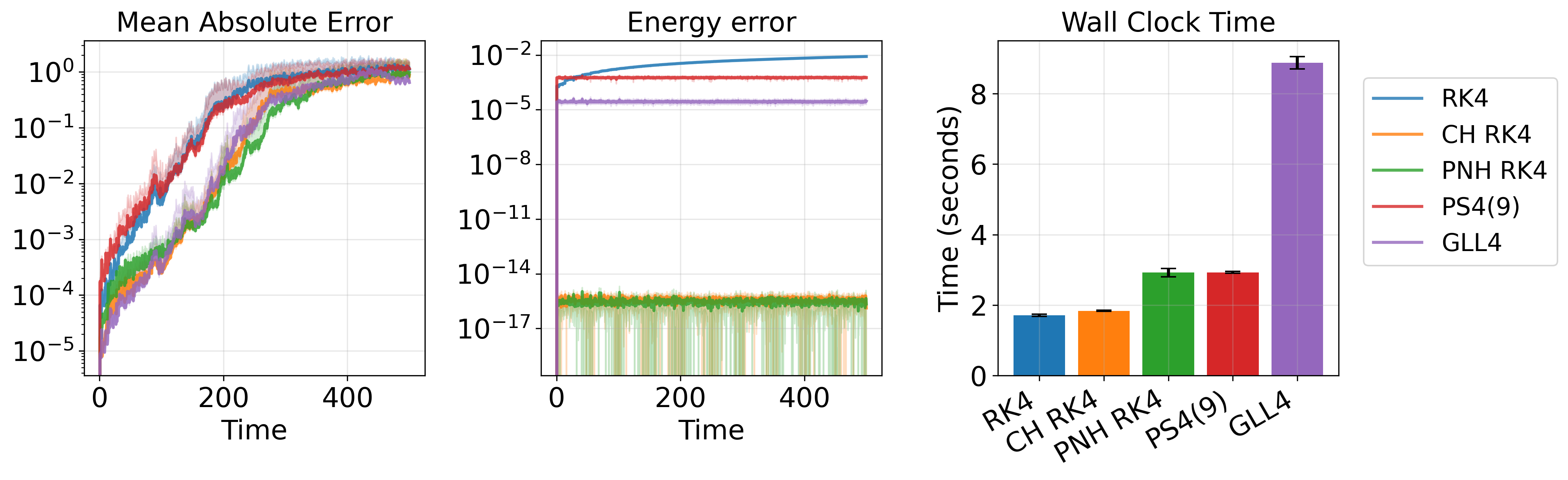} 
    \caption{Torsioned joint potential double pendulum.}
    \label{fig:dp torsioned}
  \end{subfigure} 
  \begin{subfigure}{0.8\textwidth}
    \includegraphics[width=\textwidth]{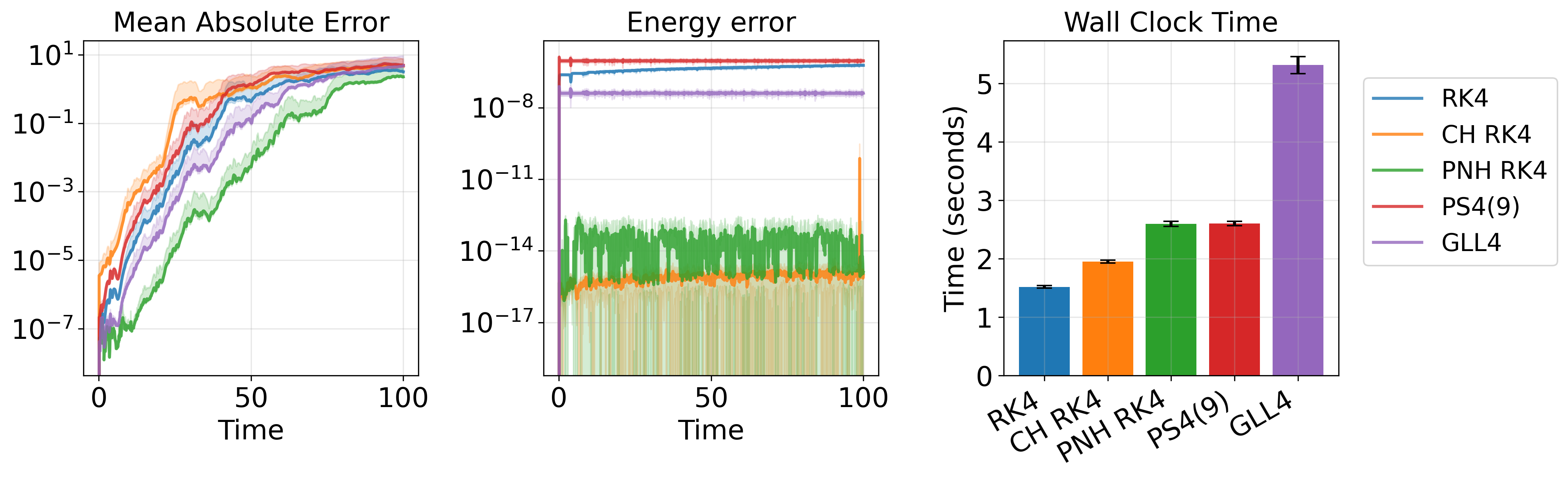} 
    \caption{Gravitational potential double pendulum.} 
    \label{fig:dp gravity} 
  \end{subfigure}
  \caption{The error, energy preservation and runtime of the double pendulum systems averaged over 10 random initial conditions. The error bars and shaded regions indicate one standard deviation.}
  \label{fig:dp-errors}
\end{figure}

\begin{figure}[ht] 
  \centering
    \includegraphics[width=0.3\textwidth]{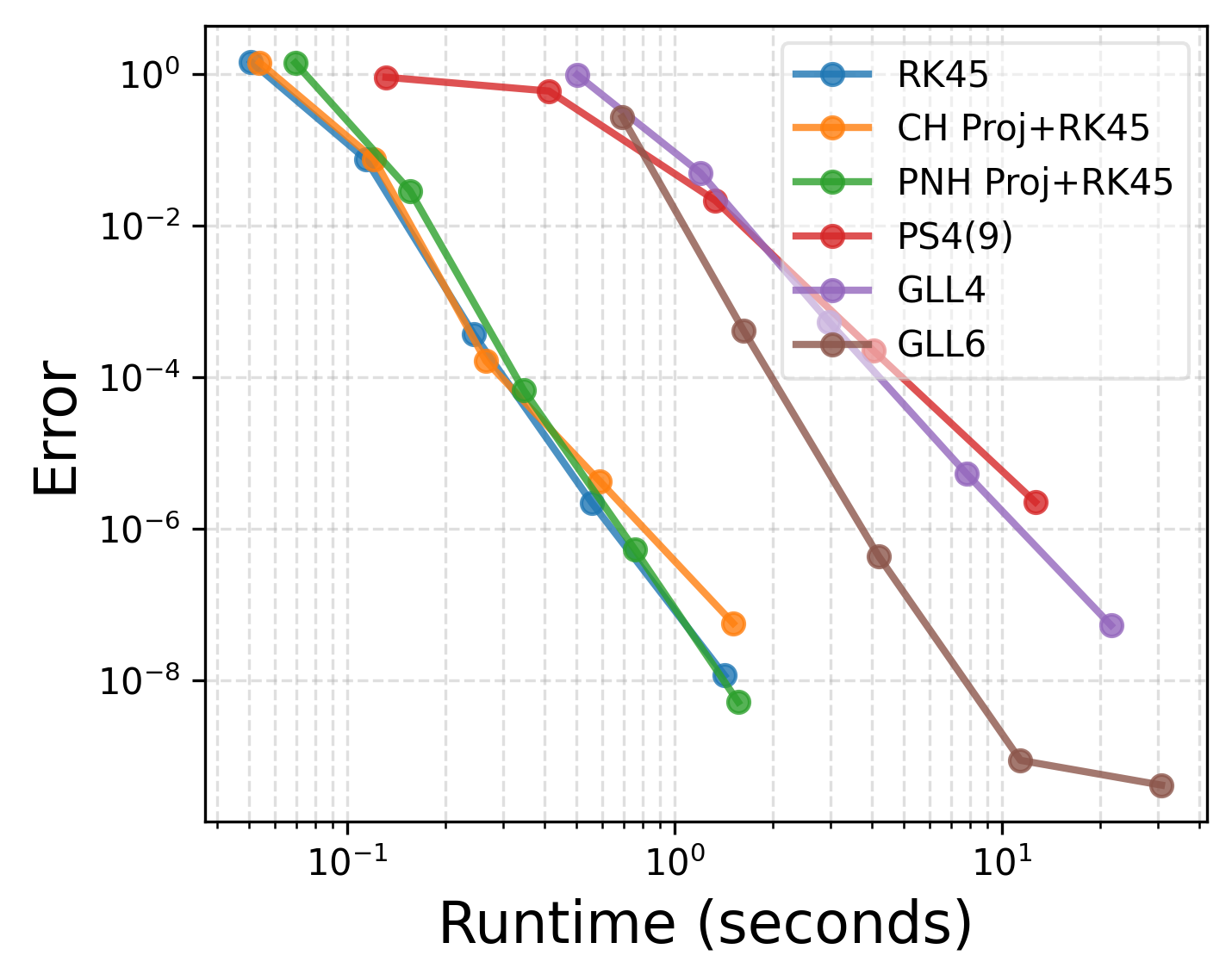} 
  \caption{The error versus runtime of adaptive time-stepping methods compared to fixed step size symplectic methods for the torsioned joint double pendulum.}
  \label{fig:dp-cost-error-adaptive}
\end{figure} 

 \paragraph{Experiment 3: adaptive-step methods for the torsioned joint potential}
Here, we consider adaptive time-stepping with the RK45 method. As adaptive symplectic methods are difficult to implement for non-separable systems like these, we include the fixed step size symplectic methods as baselines. We also include the symplectic Gauss-Lobatto-Legendre method of order 6 (GLL6) as a higher order symplectic baseline. We integrate the solution over the time interval $t\in[0,50]$ from the initial point $(q_1, q_2, p_1, p_2) = (0.1, -0.05, 1.5, -1.2)$ with various tolerances in the range $[10^{-12}, 10^{-4}]$ and time steps $h\in [10^{-3}, 10^{-1}]$. The error at the end of the interval is calculated and the total runtime is measured. These values are shown in figure \ref{fig:dp-cost-error-adaptive}. 

We see that the three adaptive methods perform the best at this timescale. The main accuracy advantages appear to come from the adaptivity, as the three adaptive RK45-based methods are roughly equivalent in terms of cost and accuracy including the non-invariant preserving RK45 method. Notably, for the same cost the adaptive methods have lower error than all the symplectic methods by several orders of magnitude. Of course, the symplectic methods have bounded energy error, which does not necessarily translate to better accuracy, especially for this timescale. The preservation of symplectic structure and energy usually manifests as better accuracy over very long time scales. So for chaotic systems like these, where small errors early on exponentially amplify, it is unsurprising that adaptivity is favoured. As such systems don't have exact solutions, it is difficult to measure long term accuracy. This is explored in the next problem. 

\subsection{Kepler problem}
The purpose of this section is to show that error controlled homogeneous projection methods, can be an effective choice of integrator even when there are fast, explicit symplectic methods available. Furthermore, we will demonstrate our projection methods for systems with multiple invariants. In this example, we consider integration of the Kepler problem, which is a two-body problem with gravitational potential in terms of the position and momentum variables $q,p\in\R^2$. Three functionally independent invariants are given by
\begin{align}
    H(q,p) &= \frac{1}{2} p^\top p - \frac{1}{\|q\|},\quad\text{(Energy)}\\ 
    L(q,p) &= q_1 p_2 - q_2 p_1,\quad\text{(Angular momentum)}\\
    A(q,p) &= p_2(p_2q_1-p_1q_2) - \frac{q_1}{\sqrt{q_1^2 + q_2^2}}, \quad\text{(Laplace-Runge-Lenz component)}
\end{align}
The energy is homogeneous of degree $k=2$ with respect to the linear group action $\psi^H_t(q,p) = (e^{-2t}q, e^{t}p)$ and $L$ is homogeneous with respect to $\psi^L_t(q, p) = (e^{at}q, e^{bt}p)$.
for free parameters $a$ and $b$. This will be used to construct an alternating projection method as in definition \ref{def:alt proj}. We consider the initial condition corresponding to an elliptical orbit with eccentricity $e\in[0, 1)$ and  $q(0) = (1-e,0)$ and $p(0) = (0,\sqrt{(1+e)/(1-e)})$. For these conditions, the period of the orbit is $2\pi$. As the eccentricity $e\rightarrow 1$, the orbit passes closer and closer to the origin (i.e., the location of the second body). This means the potential term in the energy becomes very negative $-\frac{1}{\|q\|}\rightarrow -\infty$ and therefore the kinetic term becomes large $\frac{1}{2} p^\top p \rightarrow \infty$ to keep the energy constant. On the other end of the orbit, we see the opposite effect, where dynamics is slowly varying, meaning the problem contains varying time scales making this the perfect arena to test adaptive methods. 

\subsubsection{Results}
 
We perform two experiments: 
\begin{enumerate}
  \item A cost-error analysis. Here we compare various methods over $100$ orbital periods with varying eccentricities, timesteps, tolerance and measure the error versus computational cost. Here we will show that the homogeneous projection methods are faster and more accurate than specialized geometric methods. 
  \item A long-term simulation over $10000$ orbital periods with high eccentricity. Here we will compare the solution trajectories, errors and computational cost as a function of time to measure the stability and long term performance of the methods.
\end{enumerate}

\paragraph{Description of methods} For both experiments, we will test the following methods: (1) the DOP853 method, which is the same base method that is used for the following projection methods; (2) Linear Homogeneous (LH) Projection (definition \ref{def:hom proj method}) with $\psi^H_t(q, p)=(e^{-2t}q, e^{t}p)$ to preserve $H$ exactly; (3) non-linear homogeneous (NLH) projection (theorem \ref{thm:pseudo-energy}) with a second-order Runge-Kutta method for $\Psi_1$ to preserve all three invariants $H$, $L$ and $A$ up to order $O(h^{(p+1)(q+1)})=O(h^{27})$; (4) an Alternating Linear Homogeneous (ALH) Projection according to definition \ref{def:alt proj}, where we alternate between $\psi^L_t$, $\psi^H_t$ and an order-2 PNH Projection to preserve $A$; and (5) standard orthogonal projection to preserve all three invariants exactly, solved using Newton iterations until convergence to within $10^{-13}$ or a maximum of 20 iterations. We will also compare against two more geometric integrators: (6) an explicit symplectic method of order 8 (Suzuki8) (see \cite{hairer2006geometric} for the coefficients); and (7) an explicit, symmetric, adaptive step size method (Adaptive Suzuki8) described in \cite{hairer2005explicit} using Suzuki8 as the base step integrator.

\paragraph{Experiment 1: cost-error analysis} Here, for all seven methods, we measure the mean average error as a function of cost, measured in wall clock time (runtime) after 100 orbital periods ($t\in[0, 200\pi]$) and for eccentricities $e=0.8, 0.9, 0.99$. For the DOP853 and projection methods, we vary the tolerance in the range $[10^{-14}, 10^{-4}]$. For the Suzuki8 method, we use step sizes in the range $h\in[5\times 10^{-5}, 1\times 10^{-1}]$. For the Adaptive Suzuki8 method, we vary the ``step-density" $\varepsilon\in[5\times 10^{-5}, 1\times 10^{-1}]$ (see \cite{hairer2005explicit} and \eqref{hairer-soderlind}). The results are shown in figure \ref{fig:kepler-symp}.  

We see that for all eccentricities, the homogeneous projection methods outperform all other methods in terms of accuracy and cost. The Standard Projection is similar to the homogeneous projection methods, albeit slightly less accurate and less stable. However, for high eccentricites and low tolerances, the Standard Projection is unable to attain the same accuracy as the Homogeneous Projection methods. The fixed-step symplectic method performs very well at low eccentricities, even outperforming the Adaptive Suzuki8 method. This is somewhat expected as adaptivity is not as beneficial here. However, as the eccentricity increases, the performance of the Suzuki8 method degrades significantly, becoming the least accurate and most expensive method at $e=0.99$. The Adaptive Suzuki8 method performs better, but is still slower and less accurate than the homogeneous projection methods.

We see that for all eccentricities, the homogeneous projection methods outperform all other methods in terms of accuracy and cost. The NLH projection is the most accurate method overall, while the LH projection is the fastest method overall. The Suzuki8 method performs poorly in these tesets due to the use of fixed step sizes. The Adaptive Suzuki8 method performs better, but is still slower and less accurate than the homogeneous projection methods. 

\begin{figure}[ht] 
  \centering
  \begin{subfigure}{0.245\textwidth}
    \includegraphics[width=\textwidth]{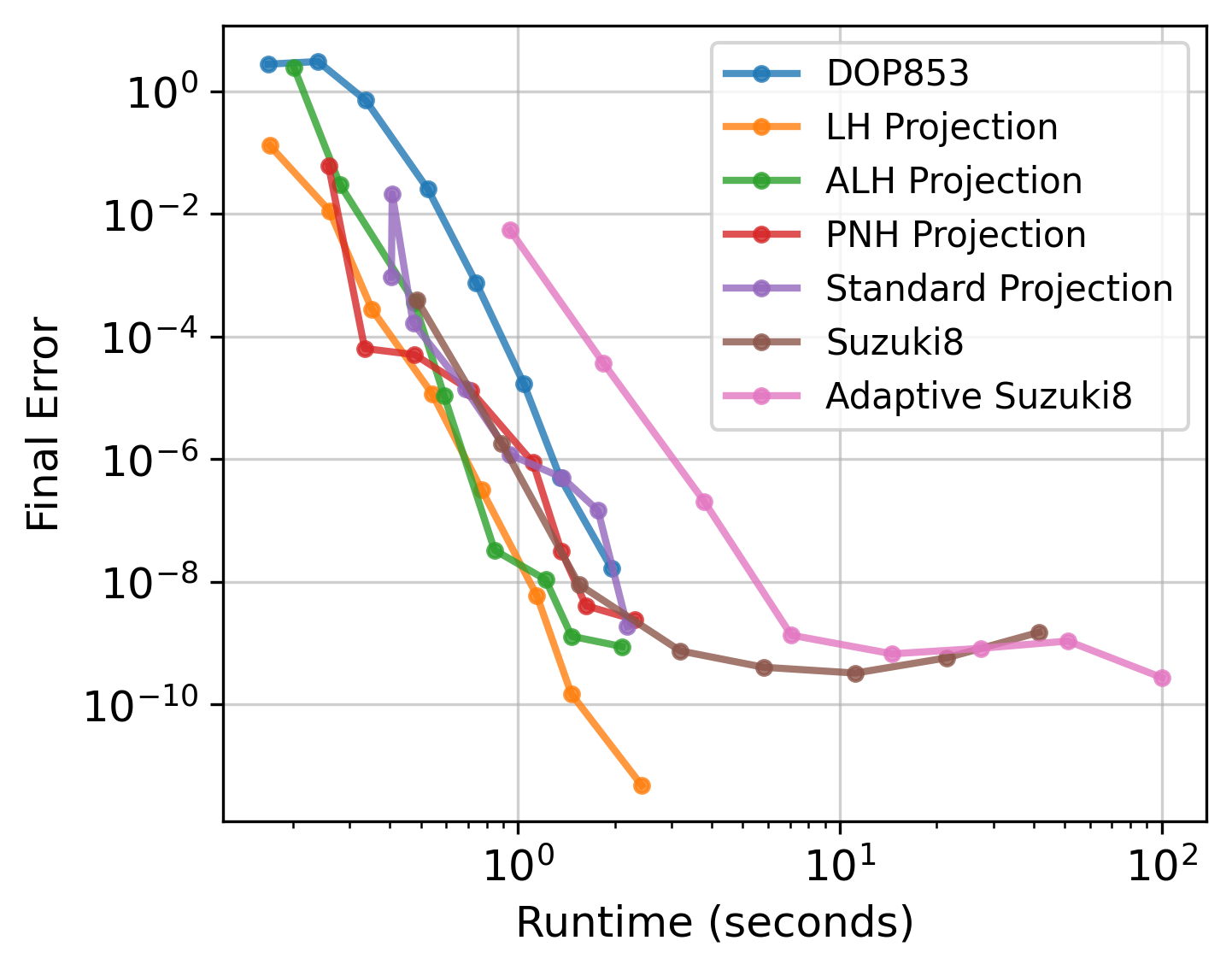} 
    \caption{eccentricity = $0.6$} 
    \label{fig:kfig0} 
  \end{subfigure} 
  \begin{subfigure}{0.245\textwidth}
    \includegraphics[width=\textwidth]{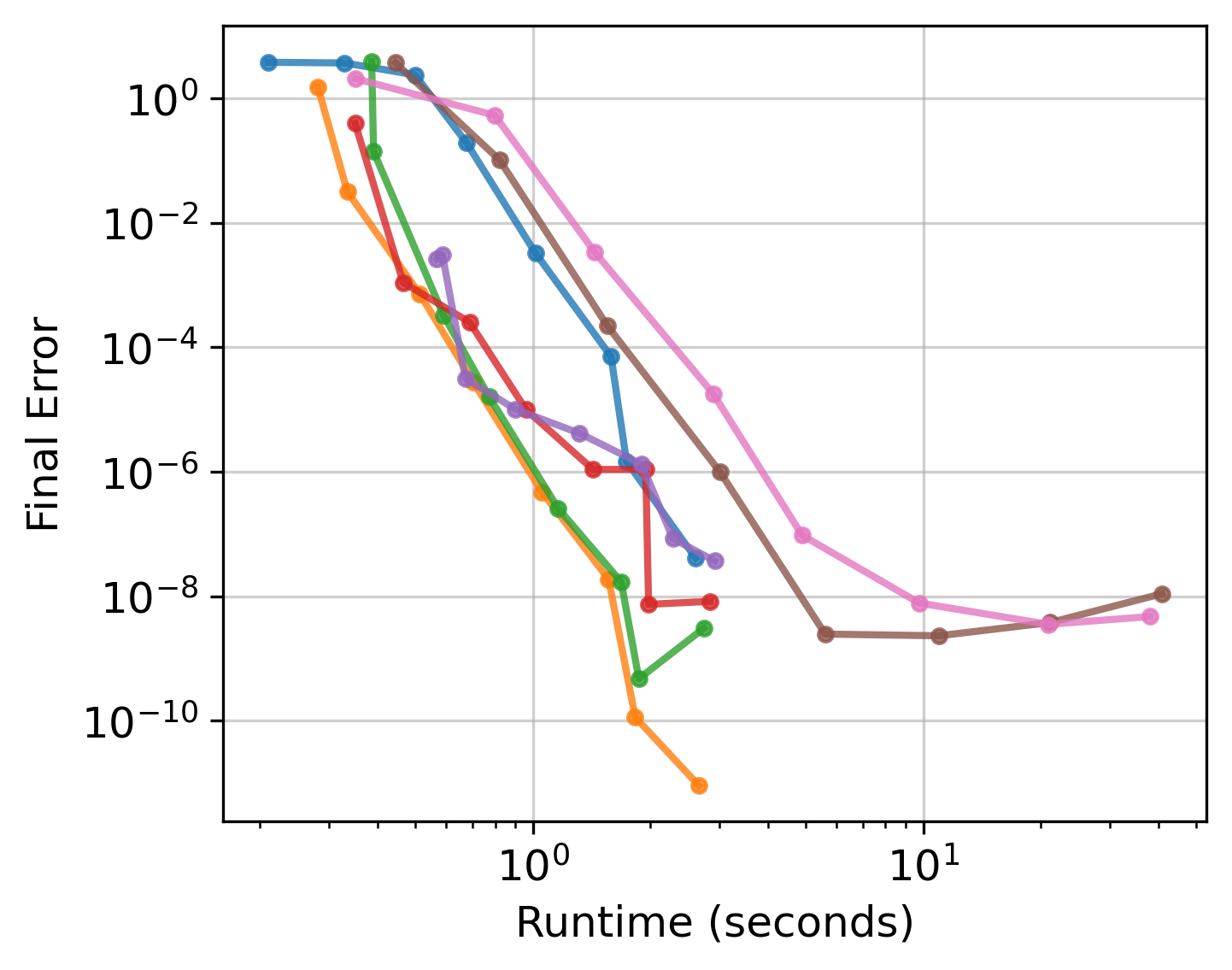} 
    \caption{eccentricity = $0.8$} 
    \label{fig:kfig1} 
  \end{subfigure} 
  \begin{subfigure}{0.245\textwidth}
    \includegraphics[width=\textwidth]{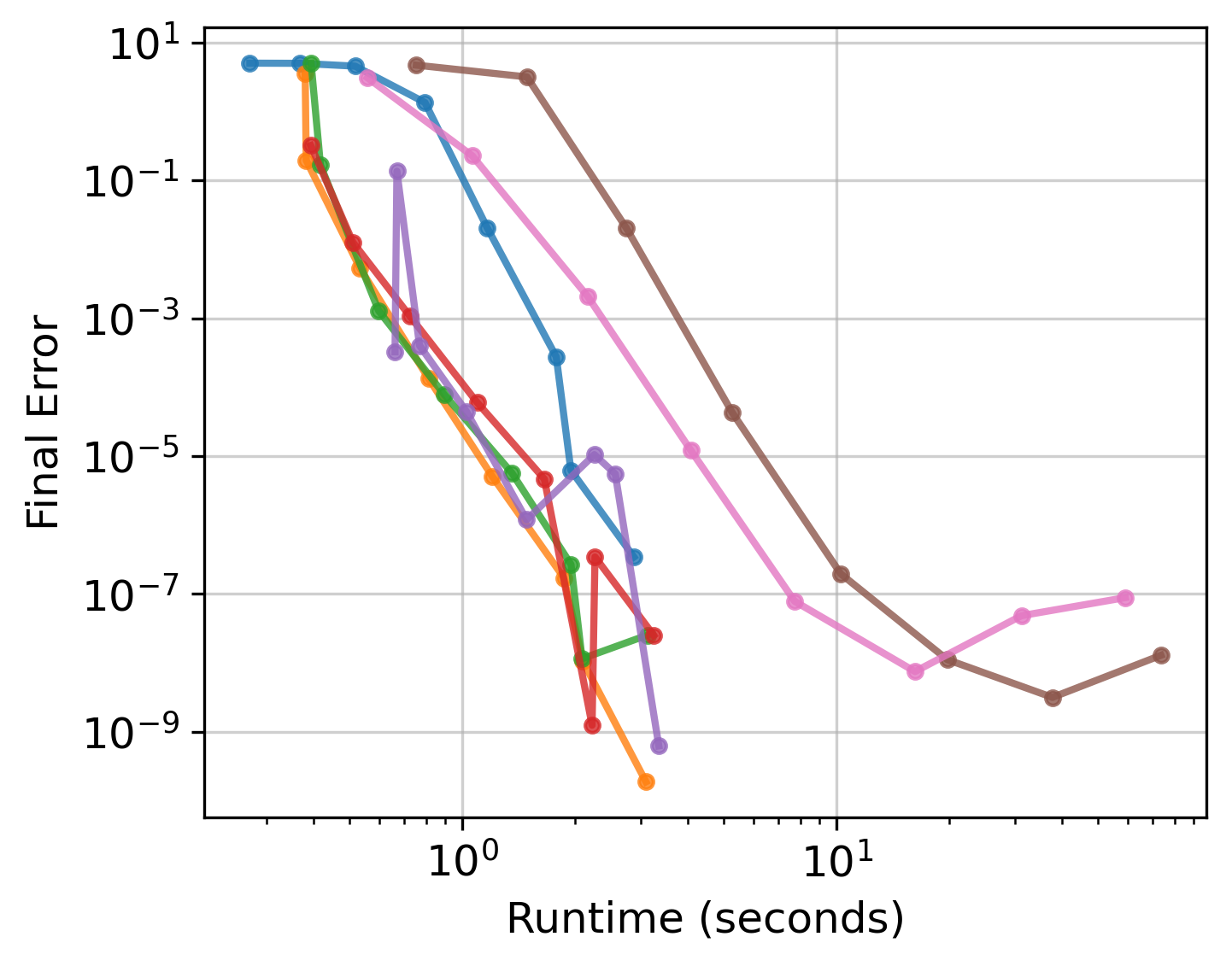} 
    \caption{eccentricity = $0.9$}
    \label{fig:kfig2}
  \end{subfigure} 
  \begin{subfigure}{0.245\textwidth}
    \includegraphics[width=\textwidth]{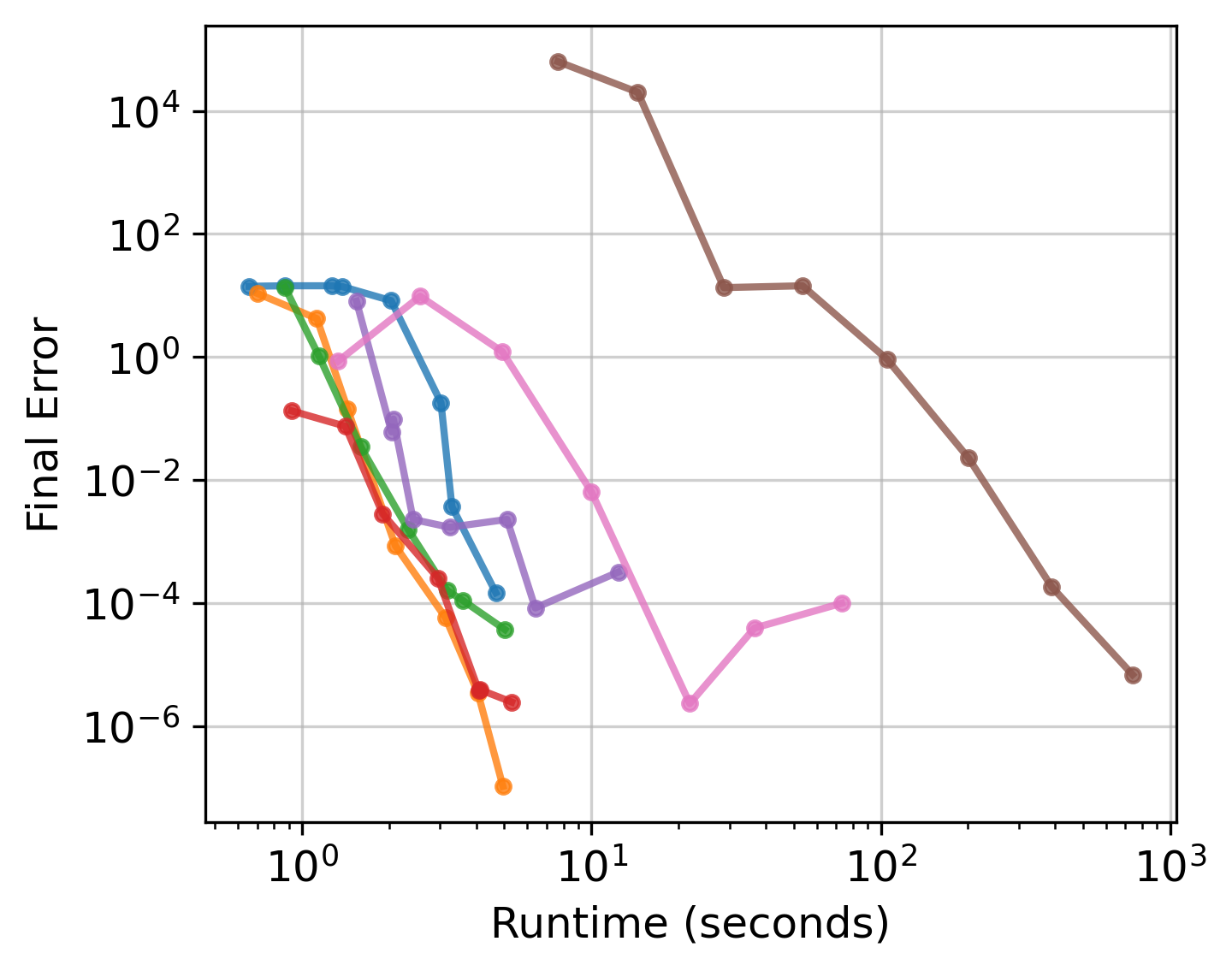} 
    \caption{eccentricity = $0.99$}
    \label{fig:kfig3}
  \end{subfigure}
  \caption{The error as a function of function evaluations (dashed lines) and runtime (solid lines) for the Kepler problem over a time interval of $[0, 1000]$ with varying eccentricities.}
  \label{fig:kepler-symp}
\end{figure}

\begin{figure}[ht] 
  \centering
  \begin{subfigure}{0.99\textwidth}
  \includegraphics[width=0.99\textwidth]{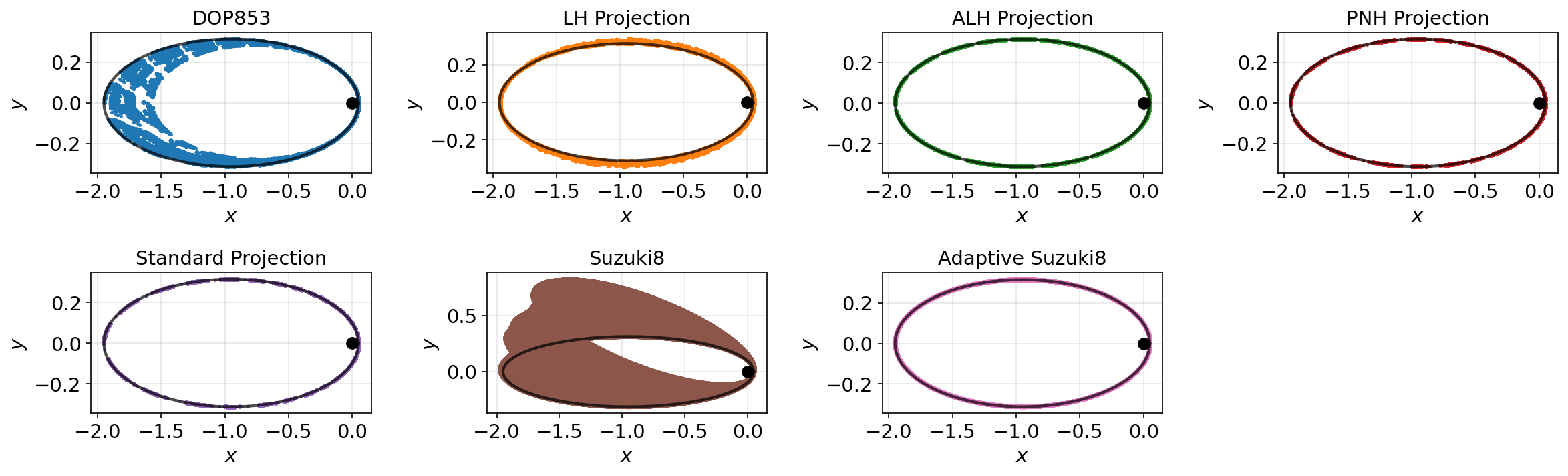} 
    \caption{Kepler orbits of the different methods. The exact solution is shown in black.}  
    \label{fig:sfig1} 
  \end{subfigure}
  \begin{subfigure}{0.99\textwidth}
    \includegraphics[width=\textwidth]{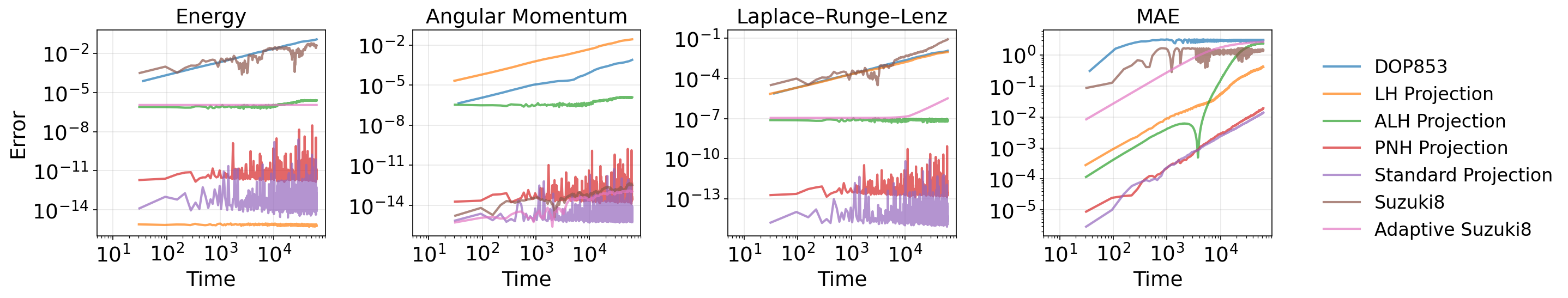} 
    \caption{Errors of the invariants and the global mean average error (MAE) of the different methods.}  
    \label{fig:sfig2} 
  \end{subfigure}
  \begin{subfigure}{0.4\textwidth}
    \includegraphics[width=\textwidth]{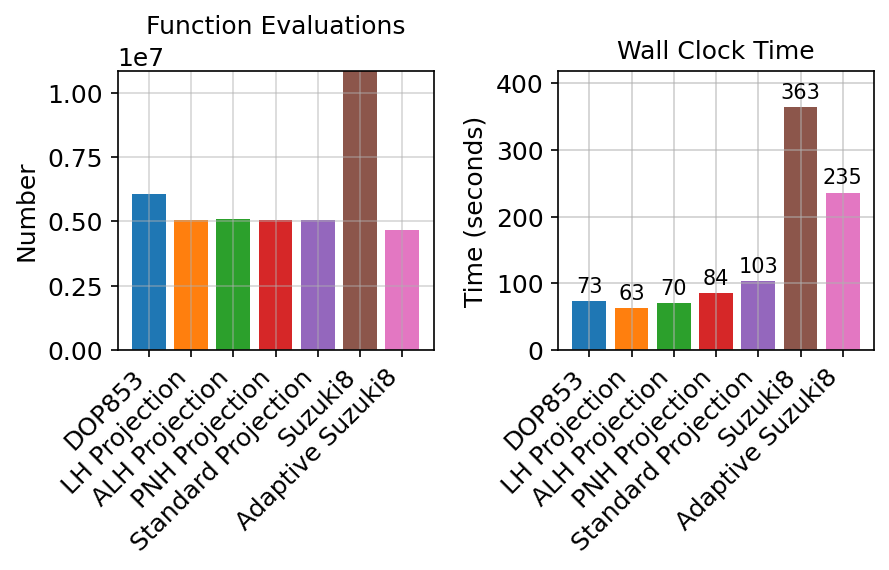} 
    \caption{Total function evaluations and runtime of each method.}
    \label{fig:sfig3}
  \end{subfigure}  
  \begin{subfigure}{0.4\textwidth}
    \includegraphics[width=\textwidth]{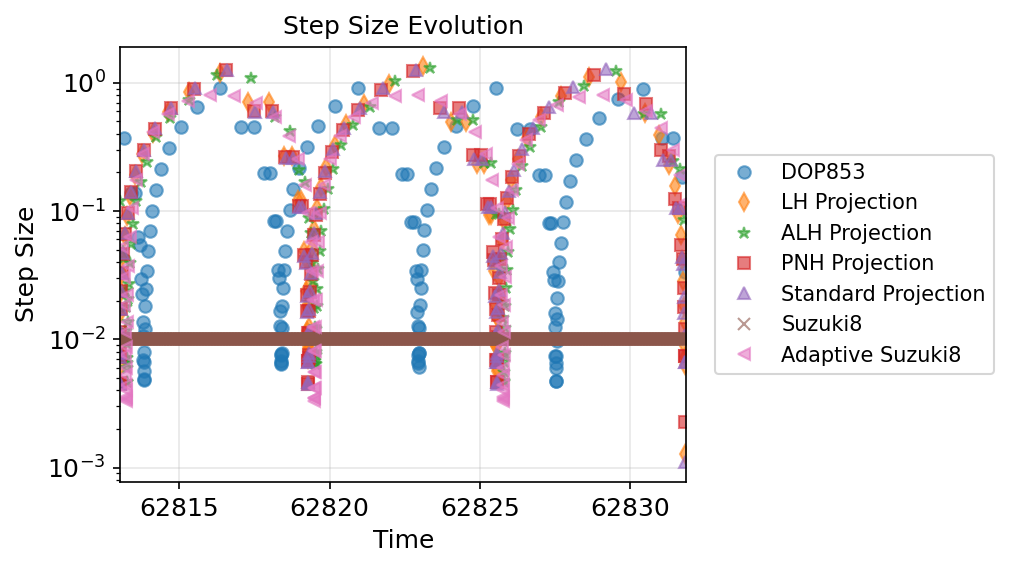} 
    \caption{Step sizes taken by each method over time (last three orbits).}
    \label{fig:sfig4}
  \end{subfigure} 
  \caption{Results for the Kepler problem with eccentricity $e=0.95$ over $10000$ periods with tolerance $10^{-6}$ and step size of $h=0.01$ for the Suzuki8 method.}
  \label{fig:kepler-errors}
\end{figure}


\paragraph{Experiment 2: long-term simulation} From figure \ref{fig:kepler-errors} we see that the all the projection methods have error about 2-4 orders of magnitude smaller than the DOP853 method for the same runtime. The error growth is linear in time for all methods, which is expected.

We see that the DOP853 method exhibits linear error growth for the invariant errors, tho at a faster rate than the projection methods. The LH projection preserves $H$ exactly, but the errors in $L$ and $A$ grow linearly. The PNH projection preserves all invariants up to roughly 10 or more digits. The LH+PNH projection preserves $H$ exactly, and $L$ and $A$ up to order $O(h^{(p+1)(q+1)})$, and the errors remain bounded over the simulation. 

The LH projection has worse $L$ and $A$ conservation despite preserving $H$ exactly. One possible explanation is that $L$ is homogeneous of degree $k=-1$ with respect to $\psi_s(q,p)$, and therefore $L(\psi_s(q,p))= L(e^{-2s}q, e^{s}p) = e^{-s}L(q,p)$ which could be adding a small bias to $L$ each time the projection is made. Despite this fact, the overall error is up to two order of magnitude better than the DOP853 method. We notice a similar effect for the LH+PNH projection, where $H$ is preserved exactly, but $L$ is not as well preserved as the PNH projection, suggesting that exact preservation of $H$ can affect the conservation of $L$. We do not see this with $A$ is invariant under $\psi_s(q,p)$. 

The invariant errors for the pseudo methods do not show any growth with time, indicating that they remain bounded near the energy preserving manifold. This suggests that newton iterations are not necessary for acheiving good qualitative behaviour and that the pseudo-invariant-preserving methods are robust enough.

In terms of cost, we see that the homogeneous projection methods actually speed up the total computation time compared to the base method DOP853. A possible explanation is that the drift from the standard method puts the solution into a more complex orbit with faster time scales, requiring smaller time steps to resolve the accuracy. This is corroborated by the fact that the DOP853 method also requires more function evaluations. The Standard Projection method requires 50\% more run time, due to the need for Newton iterations, but still produces a solution of comparable accuracy to the homogeneous projection methods. 

\section{Semi-discretised conservative PDEs}
The purpose of this section is to demonstrate that homogeneous projection on PDEs with multiple conservation laws. In particular, we will demonstrate that  even when the semidiscretization doesn't yield an ODE with an equivalent discrete conservation law (only an approximate one), that doing a pseudo-projection step to partially preserve that invariant anyway can yield much better solutions. We consider two examples: the Korteweg-de Vries (KdV) equation and the Camassa-Holm (CH) equation. Both equations are nonlinear dispersive wave equations arising in shallow water wave theory and possess multiple conservation laws. In either case, there exist straightforward ways to semidiscretise these PDEs to yield an ODE with in skew-gradient form, thus preserving exactly a discretised version at least one invariant. That is, given a PDE for $u(x, t):\Omega\times\R\rightarrow\R$ on periodic domain $\Omega=[0,L]$ with $j=1,...,m$ conservation laws of the form $\mathcal{H}_j(u)=\int_\Omega\rho_j(u, u_x,...)\rd x$, find an ODE on an $N$-dimensional equispaced grid of the form $\dot{\mathbf{u}} = S\nabla_{\mathbf{u}}H_a\in\R^N$, for some skew-symmetric $S$ and such that $u(x_i, t)\approx \mathbf{u}_i(t)$, for $i=1,...,N$ and $H_j(\mathbf{u})\approx \mathcal{H}_j(u)$. Finding semi-discretisations of this form exactly preserve $H_a$, however, in doing so, we usually only approximately preserve the other invariants when $j\ne a$. This results in errors such as aliasing or build up of truncation errors that manifest in breaking these laws over long time. We will make the claim that even if this is the case, the physical fidelity of the solution dictates that these other invariants should still be controlled, even if they are not preserved exactly by the semidiscrete ODE. To this end, we will apply pseudo-nonlinear homogeneous projections to control the growth of these other invariants, and show that this results in more accurate and stable solutions.

\subsection{KdV equation}
We consider the Korteweg-de Vries (KdV) equation, a prototypical nonlinear dispersive wave equation arising in shallow water wave theory. The KdV equation is given by
$$u_t = - 6uu_x - u_{xxx},$$
with periodic boundary conditions on the spatial domain $x\in[0, 40]$. The continuous KdV equation possesses infinitely many conservation laws, three being the mass, momentum and energy:
\begin{equation}
\mathcal{H}_1 = \int u \,dx,\quad
\mathcal{H}_2 = \int u^2 \,dx ,\quad
\mathcal{H}_3 = \int \left(u^3 - \tfrac{1}{2}u_x^2\right) dx.
\end{equation}
We discretise the spatial domain using $N=64$ grid points and centered eigth-order finite-differences to compute spatial derivatives, yielding a system of ODEs of the form 
$$\dot{\mathbf{u}} = -D\nabla_{\mathbf{u}} H_3(\mathbf{u})\in\R^N,$$
where $D$ is the skew-symmetric finite-difference differentiation matrix. Discrete analogs of the conservation laws are 
\begin{align}
H_1(\mathbf{u}) = \Delta x \sum_{i=1}^N \mathbf{u}_i,\quad
H_2(\mathbf{u}) = \Delta x \sum_{i=1}^N \mathbf{u}_i^2,\quad
H_3(\mathbf{u}) = \Delta x \sum_{i=1}^N \left(\mathbf{u}_i^3 - \tfrac{1}{2}\left(D\mathbf{u}\right)_i^2\right). 
\end{align} 
Note that only the mass and energy are exact invariants of the semi-discretized system. The momentum is only approximately conserved, with small violations arising from errors inherent to the semi-discretization. Despite this, we will still include $H_2$ in the pseudo-projection operator to help regularize the solution and prevent high frequency modes from growing uncontrollably. 
\subsubsection{Results}
We consider a single-soliton initial condition $u(x,0) = \frac{c}{2}\operatorname{sech}^2\!\left(\frac{\sqrt{c}}{2}(x - x_0)\right),$ and initial position $x_0=20$ and integrate the solution over the time interval $t\in[0,100]$. We compare three methods: (1) alternating homogeneous projections with $H_1$, $H_2$ and $H_3$, (2) the same method but only projecting $H_1$, and $H_3$, and (3) standard DOP853. We use a tolerance of $10^{-5}$ for all three. The experiment is repeated three times with increasingly challenging initial conditions by varying the soliton speed $c=2,4,8$.

The results are shown in figure \ref{fig:kdv-sols}. We see that for high soliton speeds, the DOP853 and $H_1$+$H_3$ projection becomes unstable early in the simulation, while the $H_1$,$H_2$+$H_3$ projection is more accurate and more stable throughout the simulation. Even at $c=8$ it still maintains its shape and most of its errors manifest as incorrect soliton speed. even though the momentum $H_2$ is not an exact invariant of the semi-discretized system, the projection method controls its growth and acts as a regulariser to help prevent the high frequency modes from blowing up. This could also be due to the fact that the projection method keeps the invariants bounded, preventing the solution from drifting into unstable regimes (see \ref{app: invariant preservation of PDEs}).

\begin{figure}[ht]  
  \centering
  \begin{subfigure}{0.8\textwidth}
    \includegraphics[width=\textwidth]{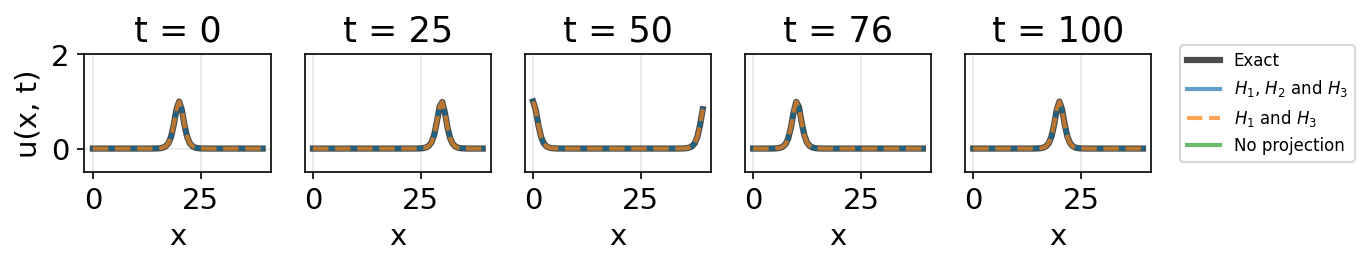} 
    \caption{Soliton speed: $c=2$.}
    \label{fig:kdv-sol1}
  \end{subfigure}
  \begin{subfigure}{0.8\textwidth}
    \includegraphics[width=\textwidth]{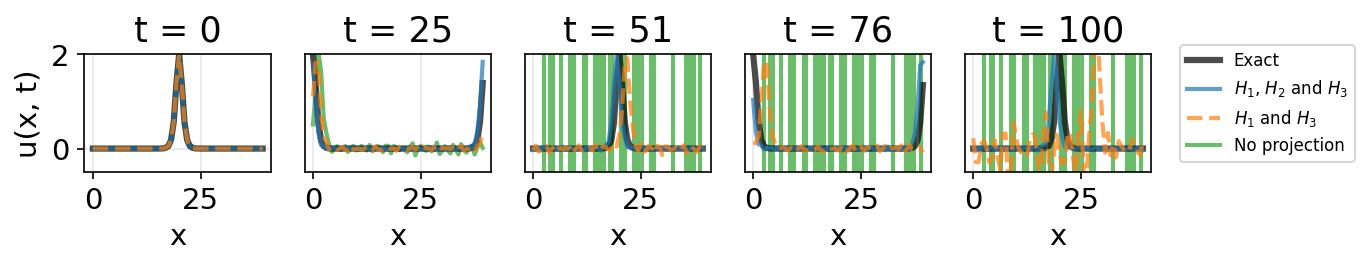} 
    \caption{Soliton speed: $c=4$.}
    \label{fig:kdv-sol2}
  \end{subfigure}
  \begin{subfigure}{0.8\textwidth}
    \includegraphics[width=\textwidth]{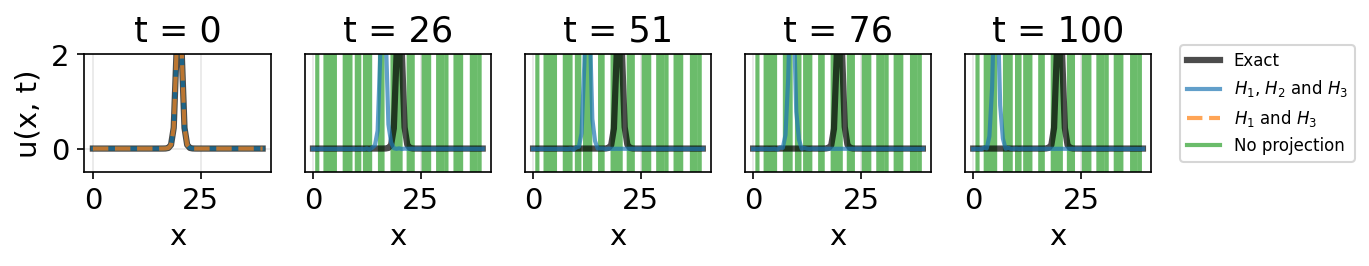} 
    \caption{Soliton speed: $c=8$.}
    \label{fig:kdv-sol3}
  \end{subfigure}
  \caption{KdV equation with increasingly challenging initial conditions.}
  \label{fig:kdv-sols}
\end{figure}
\subsection{Camassa-Holm equation}
The Camassa-Holm (CH) equation is a nonlinear dispersive wave equation that models shallow water waves. It is given by
$$u_t - u_{txx}  + 3uu_x u = 2u_x u u_{xx} + u u_{xxx},$$
with periodic boundary conditions on the spatial domain $x\in[0, 80]$. The CH equation possesses several conservation laws, including the energy and momentum
\begin{equation}
\mathcal{H}_1 = \int (u^2 + u_x^2) \,dx ,\quad
\mathcal{H}_2 = \int (u^3 + u u_x^2) \,dx.
\end{equation}
We discretize the spatial domain using $N=128$ grid points and centered fourth-order finite-differences to compute spatial derivatives, yielding a system of ODEs similar to that presented in \cite{eidnes2021linearly}
\begin{equation}
    \dot{\mathbf{u}}= (I-D_{xx})^{-1}D_x\nabla H_2(\mathbf{u}),
\end{equation}
where $D_x$ and $D_{xx}$ are the skew-symmetric finite-difference differentiation matrices of order one and two respectively. The discrete analogs of the conservation laws become
\begin{align}
H_1(\mathbf{u}) = \Delta x \sum_{i=1}^N \left(\mathbf{u}_i^2 + (D_x\mathbf{u})_i^2\right),\quad
H_2(\mathbf{u}) = \Delta x \sum_{i=1}^N \left(\mathbf{u}_i^3 + \mathbf{u}_i (D_x\mathbf{u})_i^2\right).
\end{align} 
\subsubsection{Results}
We consider the peakon initial condition $u(x,0) = c \exp(-|x - ct - x_0|)$ with $x_0=40$ and integrate the solution over the time interval $t\in[0,100]$. We compare three methods: (1) alternating homogeneous projections with $H_1$ and $H_2$, (2) the same method but only projecting $H_2$, and (3) standard DOP853. We use a tolerance of $10^{-5}$ for all three. The experiment is repeated three times with increasingly challenging initial conditions by varying the soliton speed $c=0.02,0.2,2$.

The results are shown in figure \ref{fig:ch-sols}. The main observation here is that adding the alternating pseudo-projection for $H_1$ not only stabilises the solution, but preserves the shape and speed of the solution, even for extreme initial conditions relative to the grid resolution. Furthermore, we find that all the invariant errors are bounded for the projection methods, these results are detailed in \ref{app: invariant preservation of PDEs}. 

\begin{figure}[ht]  
  \centering
  \begin{subfigure}{0.8\textwidth}
    \includegraphics[width=\textwidth]{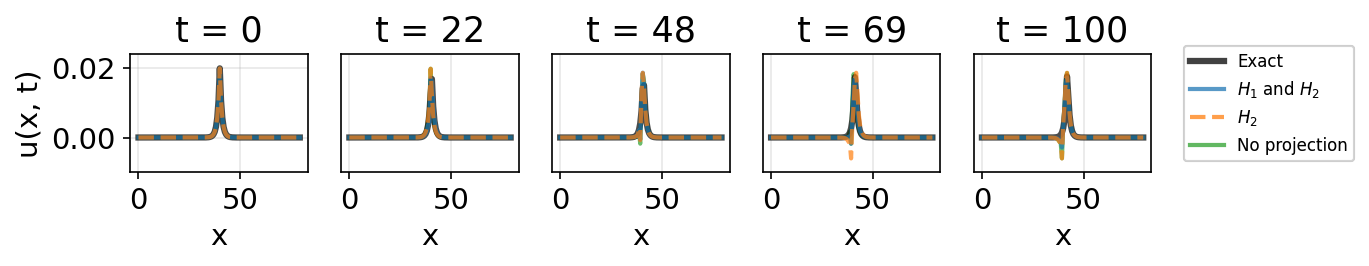} 
    \caption{Soliton speed: $c=0.02$.}
    \label{fig:ch-sol1}
  \end{subfigure}
  \begin{subfigure}{0.8\textwidth}
    \includegraphics[width=\textwidth]{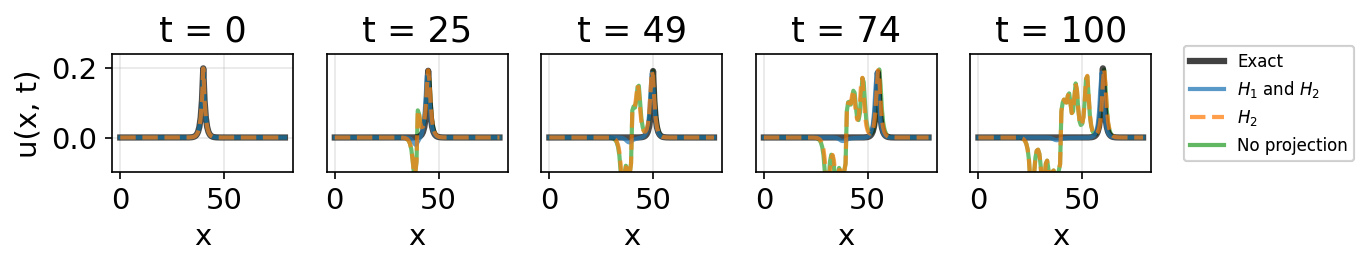} 
    \caption{Soliton speed: $c=0.2$.}
    \label{fig:ch-sol2}
  \end{subfigure}
  \begin{subfigure}{0.8\textwidth}
    \includegraphics[width=\textwidth]{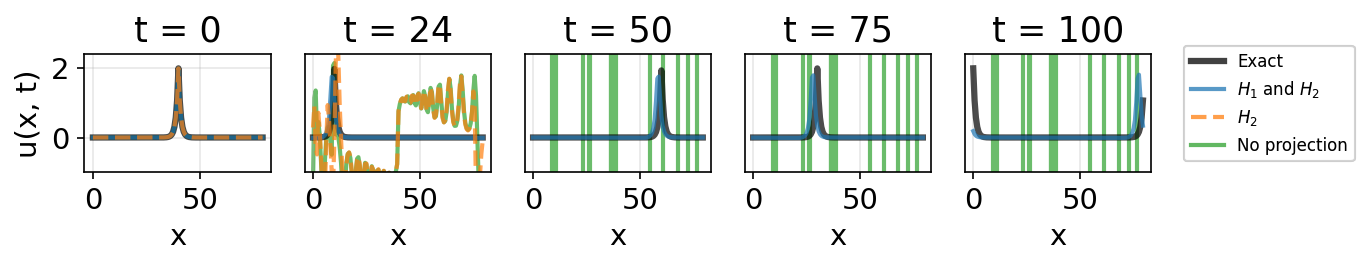} 
    \caption{Soliton speed: $c=2$.}
    \label{fig:ch-sol3}
  \end{subfigure}
  \caption{Camassa-Holm equation with increasingly challenging initial conditions.}
  \label{fig:ch-sols}
\end{figure}

\section{Conclusion}
We have presented a framework for constructing explicit invariant-preserving integrators using homogeneous projection. By leveraging homogeneous symmetries of invariants, we derived projection operators that can be evaluated in closed form, producing explicit methods that preserve invariants exactly without requiring nonlinear solves. When a closed-form homogeneous symmetry is unavailable, a pseudo-projection based on a one-step integration of an auxiliary generator preserves multiple invariants to order $\mathcal{O}(h^{(p+1)(q+1)^r})$, allowing arbitrary accuracy with minimal computational cost. We also presented an alternating projection that is more efficient and can accomodate multiple invariants. The framework naturally extends to nonlinear and conjugate symmetries, enabling its use for a wide range of systems, from finite-dimensional Hamiltonian dynamics to semi-discretized PDEs.

Across all numerical experiments, the homogeneous projection methods consistently outperform standard symplectic and adaptive integrators in both accuracy and efficiency. While symplectic schemes maintain phase-space volume only for fixed step sizes and lose structure under adaptive control, our methods remain fully compatible with adaptive step-size selection without compromising invariant preservation. This combination of structure preservation and adaptivity is a key advantage: invariant-controlled adaptivity allows the solver to automatically refine near rapid transients while coarsening in smooth regions, maintaining high accuracy at lower cost. In the double pendulum and Kepler experiments, the homogeneous projection methods achieved several orders of magnitude smaller error than comparably expensive symplectic and reversible methods. 

When it comes to semidescretised conservative PDEs, we suggested that some challenges in finding stable and conservative spatial discretisations can be overcome by a projection. This is corroborated to some extent in our numerical experiments with the KdV and Camassa-Holm equations. It would be interesting to explore this further in future work.

Overall, homogeneous projection offers a practical, general, and easily implemented route to adaptive, explicit, structure-preserving integration. It makes invariant-preservation much more attractive with adaptive schemes, providing a useful tool for simulating complex dynamical systems with high fidelity and efficiency.

\section*{Acknowledgements}
During the preparation of this work the author used ChatGPT, Deep Research and GitHub Copilot in order to improve clarity of text and presentation and for writing code. The author reviewed and edited the content as needed and takes full responsibility for the content of the published article.

We would like to thank the Research Council of Norway for funding the PhysML project (338779).
\bibliographystyle{plain}
\bibliography{bibliography}

\appendix
\section{Convergence of pseudo-invariant-preserving homogeneous projection methods}\label{app:energy convergence}
In this section we verify the convergence rate of the pseudo non-linear homogeneous projection method described in theorem \ref{thm:pseudo-energy} and corollary \ref{cor:superlinear}. We consider a non-linear oscillator in four dimensions with Hamiltonian
\begin{equation}\label{eq:nonlinear oscillator}
  H = (p_1^2 + p_2^2)/2 + 3(0.5q_1^4 + q_2^4) + 6(q_1^2 + 2q_2^2) + 2q_1q_2(q_1^2 + 2q_2^2) + 3\sin(5q_1)\cos(3q_2)
\end{equation}
and measure the energy error after one time step of size for varying $h$. Letting $\Phi_h^{[p]}$ be an order $p$ one step method and $\Psi_h^{[q]}$ be an order $q$ projection method. Then we test the energy preservation of $\widehat{\Phi}_h^{p,q,r}= \left(\Psi_h^{[q]}\right)^{r}\circ\Phi_h^{[p]}$ for $p = 1, 2, 4$, $q=1, 2, 4$ and $r=1, 2, 3$. The results are shown in figure \ref{fig:convergence}, where we see that the energy error scales as $O\big(h^{(p+1)(q+1)^r}\big)$ as predicted by corollary \ref{cor:superlinear}.

\begin{figure}[ht]
  \centering
  \includegraphics[width=0.9\textwidth]{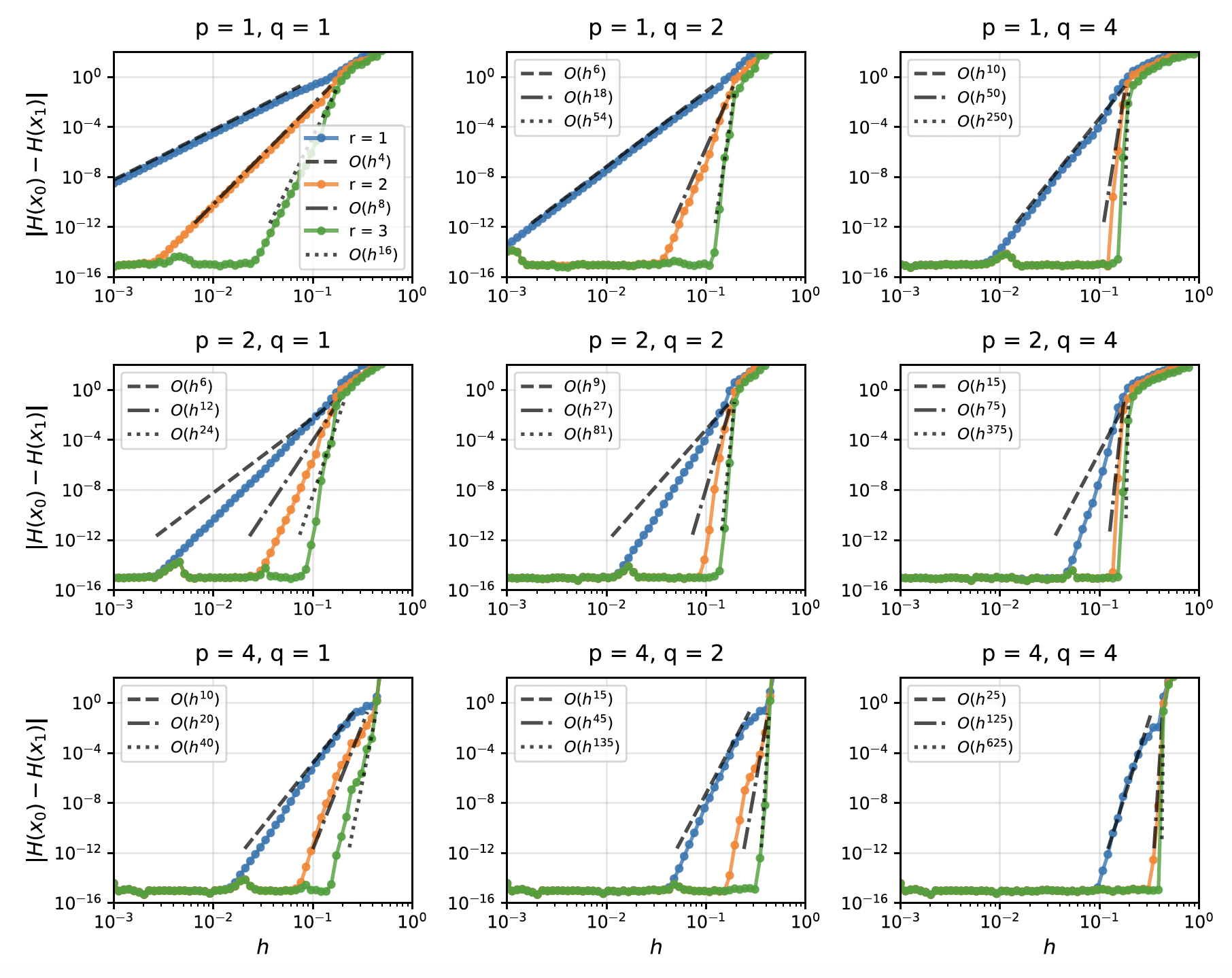} 
  \caption{Convergence analysis of the pseudo non-linear homogeneous projections for the non-linear oscillator in four dimensions. The black lines indicate the theoretical convergence rates, the blue, orange and green lines indicate the $r=1, 2, 3$ methods, respectively.}
  \label{fig:convergence}
\end{figure}

\section{Additional figures for KdV and Camassa-Holm experiments}\label{app: invariant preservation of PDEs}
Here we present additional figures for the KdV and Camassa-Holm experiments showing the invariant preservation of the different methods. Figures \ref{fig:kdv-invs} and \ref{fig:ch-invs} show that the alternating homogeneous projection method keeps the invariants bounded throughout the simulation, while the other methods exhibit drift and instability over these time intervals.

\begin{figure}[ht]  
  \centering
  \begin{subfigure}{0.6\textwidth}
    \includegraphics[width=\textwidth]{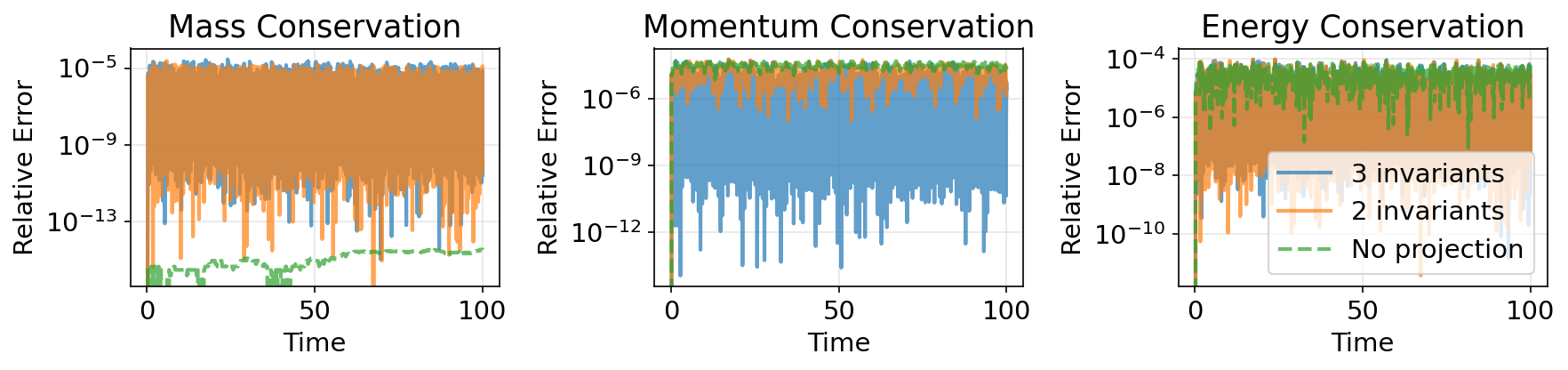} 
    \caption{Soliton speed: $c=2$.}
    \label{fig:kdv-inv1}
  \end{subfigure}
  \begin{subfigure}{0.6\textwidth}
    \includegraphics[width=\textwidth]{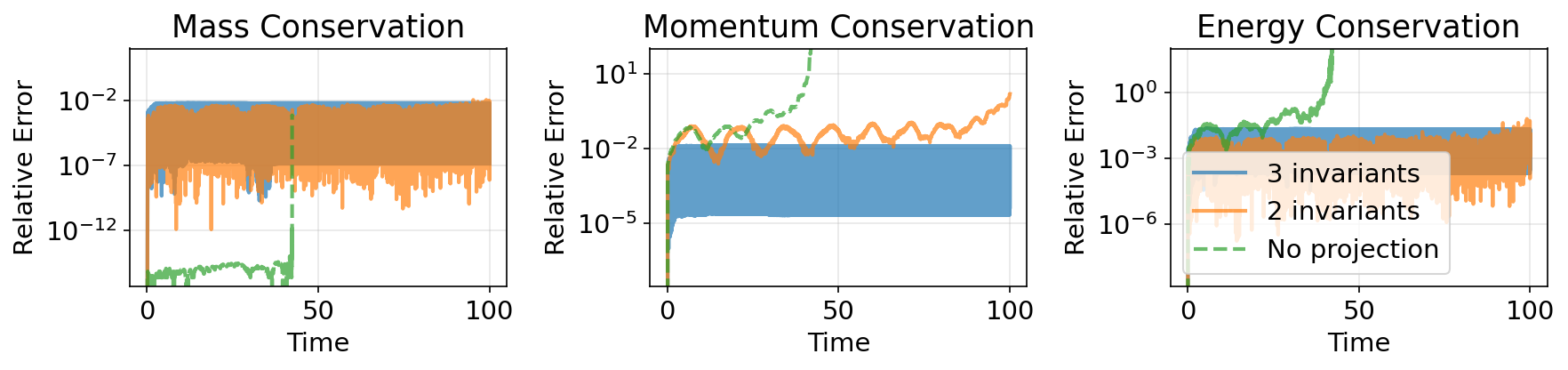} 
    \caption{Soliton speed: $c=4$.}
    \label{fig:kdv-inv2}
  \end{subfigure}
  \begin{subfigure}{0.6\textwidth}
    \includegraphics[width=\textwidth]{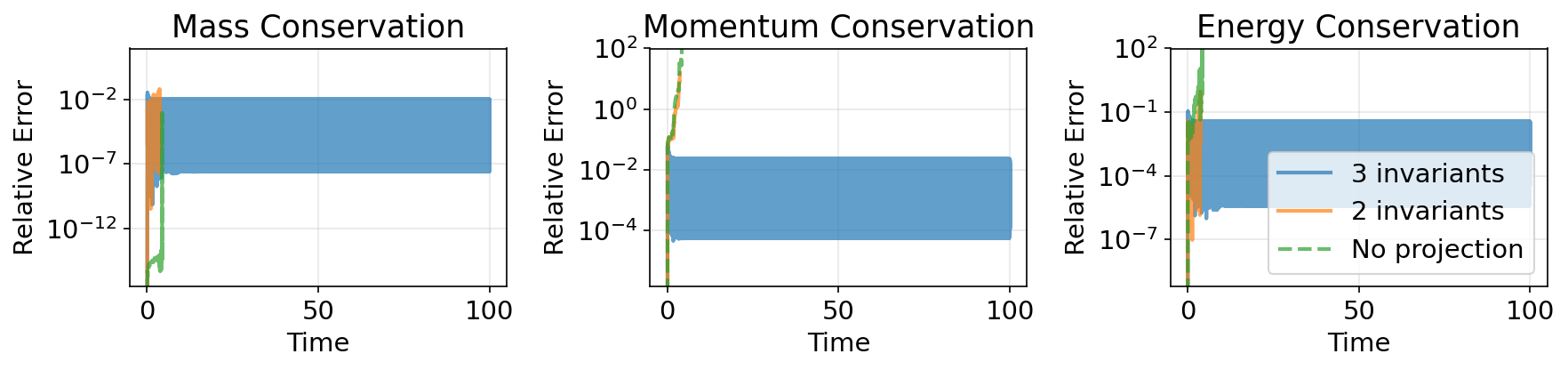} 
    \caption{Soliton speed: $c=8$.}
    \label{fig:kdv-inv3}
  \end{subfigure}
  \caption{KdV equation invariant preservation with increasingly challenging initial conditions.}
  \label{fig:kdv-invs}
\end{figure}

\begin{figure}[ht]   
  \centering
  \begin{subfigure}{0.4\textwidth}
    \includegraphics[width=\textwidth]{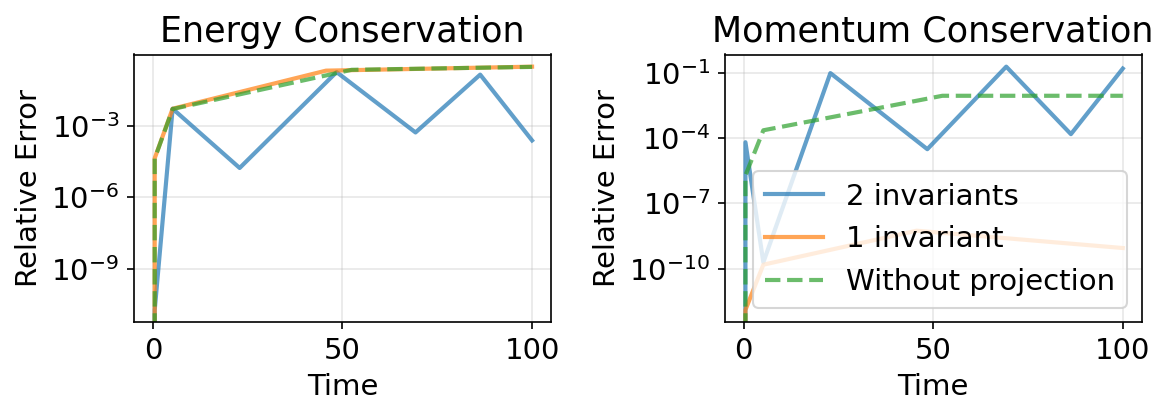} 
    \caption{Soliton speed: $c=0.02$.}
    \label{fig:ch-inv1}
  \end{subfigure}

  \begin{subfigure}{0.4\textwidth}
    \includegraphics[width=\textwidth]{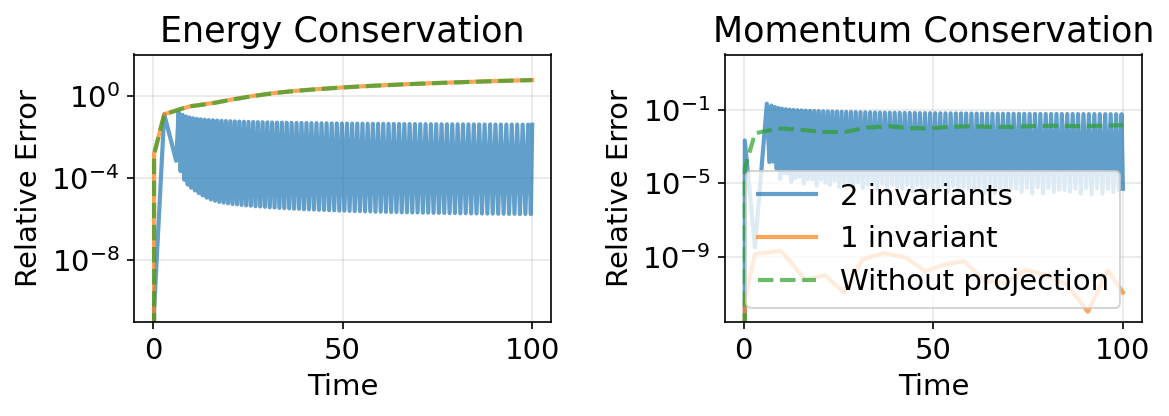} 
    \caption{Soliton speed: $c=0.2$.}
    \label{fig:ch-inv2}
  \end{subfigure}

  \begin{subfigure}{0.4\textwidth}
    \includegraphics[width=\textwidth]{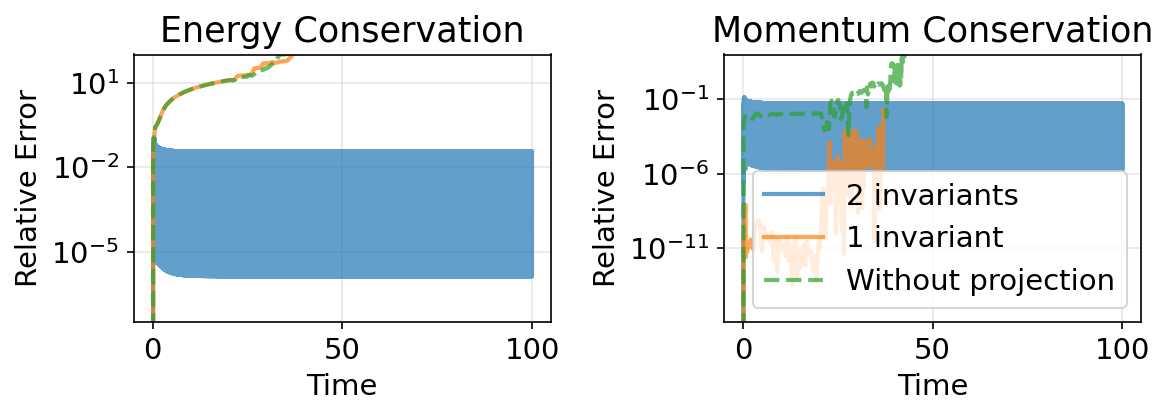} 
    \caption{Soliton speed: $c=2$.}
    \label{fig:ch-inv3}
  \end{subfigure}
  \caption{Camassa-Holm equation invariant preservation with increasingly challenging initial conditions.}
  \label{fig:ch-invs}
\end{figure}

\end{document}